\documentclass[11pt]{amsart}
\usepackage{amssymb,amsmath,amsthm,amsrefs}
\usepackage[mathscr]{eucal}
\usepackage{fullpage}
\usepackage{enumerate}
\usepackage{verbatim}
\usepackage{graphicx,float}

\usepackage{color}

\usepackage{psfrag}

\newtheorem{theorem}[equation]{Theorem}
\newtheorem{lemma}[equation]{Lemma}
\newtheorem{prop}[equation]{Proposition}
\newtheorem{cor}[equation]{Corollary}
\newtheorem{corollary}[equation]{Corollary}

\newtheorem{definition}[equation]{Definition}

\theoremstyle{remark}
\newtheorem{remark}[equation]{Remark}
\newtheorem{notation}[equation]{Notation}

\numberwithin{equation}{section}

\newcommand{\R}{\mathbb{R}}

\newcommand{\Z}{\mathbb{Z}}
\newcommand{\N}{\mathbb{N}}
\newcommand{\Sph}{\mathbb{S}}

\newcommand{\ind}{\operatorname{Ind}}
\newcommand{\nul}{\operatorname{Null}}

\newcommand{\stab}{\operatorname{Stab}}
\newcommand{\Grp}{\mathscr{G}} 

\newcommand{\Gsym}{\mathscr{G}_{sym}}

\newcommand{\Fcal}{\mathcal{F}}

\newcommand{\Lcal}{\mathcal{L}}
\newcommand{\Tr}{{\mathsf{T}}}

\newcommand{\pp}{\mathsf{p}}
\newcommand{\pqq}{\mathsf{q}}
\newcommand{\QQ}{\mathsf{x}}
\newcommand{\pyy}{\mathsf{y}}
\newcommand{\rot}{\mathsf{R}}

\newcommand{\refl}{{\underline{\mathsf{R}}}}
\newcommand{\oplusu}{{\,{\oplus}_{\Lcal}\,}} 
\newcommand{\Om}{\Omega} 
\newcommand{\Omu}{{\underline{\Omega}}} 
\newcommand{\Uu}{{\underline{U}}} 
\newcommand{\gu}{{\underline{g}}} 
\newcommand{\gammau}{{\underline{\gamma}}} 
\newcommand{\Lcalu}{{\underline{\Lcal}}} 
\newcommand{\Uuu}{{\underline{U}'}} 
\newcommand{\Lcaluu}{{\underline{\Lcal}'}}

\newcommand{\norm}[1]{\left\|#1\right\|}

\newcommand{\cone}{\mbox{$\times \hspace*{-0.244cm} \times$}}
\newcommand{\Span}{\operatorname{Span}}
\newcommand{\Ktilde}{\widetilde{K}}  
\newcommand{\gammatilde}{\gamma'}  

\newcommand{\dd}{{\boldsymbol{\partial}\mspace{.8mu}\!\!\!\!\boldsymbol{/}\,}}
\newcommand{\Qv}{{{Q\mspace{.8mu}\!\!\!\!{/}\,}}}
\newcommand{\Dcal}{\dd}

\title{The index and nullity of the Lawson surfaces $\xi_{g,1}$}

\author{Nikolaos Kapouleas}
\address{Department of Mathematics, Brown University, Providence, RI 02912}
\email{nicos@math.brown.edu}

\author{David Wiygul}
\address{Department of Mathematics and Statistics, California State University, Long Beach, CA 90840} 
\email{david.wiygul@csulb.edu}

\date{\today}

\begin{document}

\begin{abstract}
We prove that the Lawson surface $\xi_{g,1}$ 
in Lawson's original notation, 
which has genus $g$ and can be viewed as a desingularization of two orthogonal great two-spheres in the round three-sphere $\Sph^3$,   
has index $2g+3$ and nullity $6$ for any genus $g\ge2$. 
In particular $\xi_{g,1}$ has no exceptional Jacobi fields, 
which means that it cannot ``flap its wings'' at the linearized level 
and is $C^1$-isolated.   
\end{abstract}

\maketitle

\section{Introduction}

\subsection*{The general framework and brief discussion of the results}
$\phantom{ab}$
\nopagebreak

Determining the index and nullity of complete  or closed minimal surfaces is a difficult problem 
which has been fully solved only in a few cases; 
see for example \cites{nayatani1992,nayatani1993, morabito}. 
The index plays an important role in min-max theory \cite{neves2014};  
this provides partial motivation for our result.  
In this article we prove Theorem \ref{Mtheorem}, 
which determines (for the first time) the index and the nullity of the Lawson surfaces $\xi_{g,1}$ \cite{Lawson} with $g\ge2$. 
These are the Lawson surfaces which have genus $g$ and can be viewed as desingularizations of two orthogonal great two-spheres in the round three-sphere $\Sph^3$ 
in the sense of \cite{alm20}*{Definition 1.3}. 
The index determined is consistent with (but larger than) a lower bound established by Choe \cite{choe1990}. 
We prove that the nullity is $6$ and so there are no exceptional Jacobi fields, 
which means by Corollary \ref{C:flapping} that these surfaces cannot ``flap their wings'' at the linearized level 
and are $C^1$-isolated.   
This provides a partial answer to questions asked in \cite{alm20}*{Section 4.2}.   

The ideas of our proof originate with work of NK on the approximate kernel for Scherk surfaces \cites{compact, alm20}.  
Our approach requires a detailed understanding of the elementary geometry of $\Sph^3$ and of the surfaces involved, 
especially their symmetries. 
The proof makes heavy use also of Alexandrov reflection in the style of Schoen's \cite{schoen1983}. 
The Courant nodal theorem \cite{courant} and an argument of Montiel-Ros \cite{Ros} play essential roles as well. 
In ongoing work we hope to extend this result to determine the index and nullity 
of all Lawson surfaces 
$\xi_{m-1,k-1}$ in Lawson's original notation, with $m\ge k\ge 3$. 

Another interesting problem, which could not be posed until the determination of the index of the Lawson surfaces,  
is motivated by the characterization of the Clifford torus by Fischer-Colbrie (unpublished) and (independently) by Urbano \cite{urbano}  
as the only closed minimal surface in $\Sph^3$, besides the great sphere, which has index $\le5$, 
and also by some recent results for minimal surfaces in $\R^3$ \cites{davi,davi2}: 
the problem is to classify all closed minimal surfaces in $\Sph^3$ which have index $\le7$ 
(the index of the Lawson surface of genus two), 
or more generally $\le2g+3$ for small $g$ 
(the index of the Lawson surface $\xi_{g,1}$).

\subsection*{Notation and conventions}
$\phantom{ab}$
\nopagebreak

We denote by $\Sph^3 \subset \R^{4}$ the unit $3$-dimensional sphere. 

\begin{notation} 
\label{span} 
For any 
$A\subset \Sph^3 \subset \R^{4}$ 
we denote by $\Span(A)$ the span of $A$ as a subspace of $\R^{4}$  
and by $\Sph(A):=\Span(A)\cap\Sph^3$. 
\qed 
\end{notation} 

Given now a vector subspace $V$ of the Euclidean space $\R^{4}$, 
we denote by $V^\perp$ its orthogonal complement in $\R^{4  }$,  
and we define the reflection in $\R^{4}$ with respect to $V$,   
$\refl_V: \R^{4} \to \R^{4} $, by 
\begin{equation} 
\label{reflV} 
\refl_V:= \Pi_V - \Pi_{V^\perp}, 
\end{equation}  
where $\Pi_V$ and $\Pi_{V^\perp}$ are the orthogonal projections of $\R^{4}$ onto $V$ and $V^\perp$ respectively. 
Alternatively 
$\refl_V: \R^{4  } \to \R^{4  } $ is the linear map which restricts to the identity on $V$ 
and minus the identity on $V^\perp$.  
Clearly the fixed point set of $\refl_V$ is $V$. 

\begin{definition}[Reflections $\refl_A$] 
\label{D:refl} 
Given any 
$A\subset  \Sph^3 \subset \R^{4  }$,  
we define $A^\perp:=\left(\,Span(A) \, \right)^\perp \cap \Sph^3 $  
and  
$\refl_A : \Sph^3 \to \Sph^3 $ to be the restriction to $\Sph^3$ of $\refl_{\Span(A)}$. 
Occasionally we will use simplified notation:  
for example for $A$ as before and $p\in\Sph^3$ we may write $\Sph(A,p)$ and $\refl_{A,p}$ instead of $\Sph(A\cup\{p\})$ and $\refl_{A\cup\{p\}}$ 
respectively. 
\qed 
\end{definition} 

Note that the set of fixed points of $\refl_A$ above is $\Sph(A)$ as in notation \ref{span},  
which is $\Sph^3$, 
or a great two-sphere,  
or a great circle,  
or the set of two antipodal points,  
or the empty set, 
depending on the dimension of $\Span(A)$. 
Following now the notation in \cite{choe:hoppe}, we have the following. 

\begin{definition}[The cone construction] 
\label{D:cone} 
For $p,q\in\Sph^3$ which are not antipodal we denote 
the minimizing geodesic segment joining them by $\overline{pq}$. 
For $A,B\subset\Sph^3$ such that no point of $A$ is antipodal to a point of $B$ 
we define the cone of $A$ and $B$ in $\Sph^3$ by 
$$ 
A\cone B := \bigcup_{p\in A, \, q\in B} \overline{pq}. 
$$ 
If $A$ or $B$ contains only one point we write the point instead of $A$ or $B$ respectively; 
we have then $p\cone q = \overline{pq} $ 
for any 
$p,q\in\Sph^3$ which are not antipodal. 
More generally, 
given linearly independent $p_1,\cdots,p_k\in \Sph^3$, 
we define inductively for $k\ge3$ 
$
\overline{p_1\cdots p_k} := p_k \cone \overline{p_1\cdots p_{k-1} }.  
$
\qed 
\end{definition}

If $\Grp$ is a group acting on a set $B$ and if $A$ is a subset of $B$,
then we refer to the subgroup 
  \begin{equation}
  \label{stab}
    \stab_{\Grp}(A):=\{ \mathbf{g} \in \Grp \; | \; \mathbf{g}A = A \}
  \end{equation}
as the \emph{stabilizer} of $A$ in $\Grp$.
When $A$ is a subset of the round $3$-sphere, we will set
  \begin{equation}
  \label{Gsym}
    \Gsym^A:=\stab_{{O(4)}} A. 
  \end{equation}

In the next definition we find it convenient to work with piecewise-smooth functions on a domain in a surface. 
By this we mean that each such function is continuous on the domain, 
the domain can be subdivided into domains by a finite union of piecewise-smooth embedded curves, 
and on the closure of each of these domains the function is smooth. 
We use $C_{pw}^\infty(\Uu)$ to denote the space of piecewise-smooth functions on a domain $\Uu$.

\begin{definition}[Eigenvalues] 
\label{D:eigen} 
We assume given a compact domain $\Uu$ in a smooth surface equipped with a Riemannian metric $\gu$,   
a smooth function $f$ on $\Uu$,  
and a linear space of piecewise-smooth functions $V'\subset C_{pw}^\infty(\Uu)$ which is invariant under the Schr\"{o}dinger operator 
$\Lcalu=\Delta_{\gu}+f$
defined on $\Uu$. 
We define 
$\lambda_i(V',\Lcalu)$ to be the $i^{th}$ eigenvalue,  
where we are counting in non-decreasing order and with multiplicity. 
(Note also that we follow the conventions which make the eigenvalues of the Laplacian on a closed surface nonnegative.) 
Moreover for $\lambda\in\R$ we denote by 
$\#_{<\lambda}(V',\Lcalu )$,   
$\#_{=\lambda}(V',\Lcalu )$,   
and 
$\#_{\le\lambda}(V',\Lcalu )$,  
the number of eigenvalues $\lambda_i(V',\Lcalu)$ which are $<\lambda$, or $=\lambda$, or $\le\lambda$, respectively. 
We also define the \emph{index of $\Lcalu$ on $V'$}, 
$\ind(V',\Lcalu ):= \#_{<0}(V',\Lcalu)$, 
and 
the \emph{nullity of $\Lcalu$ on $V'$}, 
$\nul(V',\Lcalu):= \#_{=0}(V',\Lcalu)$.  
Finally note that we may omit $\Lcalu$ from the notation when it can be inferred from the context. 
\qed 
\end{definition} 

\begin{definition}[Eigenvalue equivalence] 
\label{D:eigen-eq} 
Suppose $\Lcalu$, $\Uu$, and $V'$ are as in \ref{D:eigen} 
and $\Lcaluu$, $\Uuu$, and $V''$ satisfy correspondingly the same conditions.  
We define $V'\sim_{\Lcalu, \Lcaluu} V''$---or $V'\sim V''$ if the operators are understood from the context---to mean 
that there is a linear isomorphism $\Fcal:V'\to V''$ such that the following holds: 
$\forall f'\in V'$, $f'$ is an eigenfunction with respect to $\Lcalu$ if and only if 
$\Fcal(f')$ is an eigenfunction with respect to $\Lcaluu$ of the same eigenvalue as $f$. 
We say then that $\Lcalu$ on $V'$ and $\Lcaluu$ on $V''$ are \emph{eigenvalue equivalent}.  
\qed 
\end{definition} 

Note that clearly if \ref{D:eigen-eq} holds, 
then $\forall i \in \N$ we have $\lambda_i(V' , \Lcalu ) =  \lambda_i(V'' , \Lcaluu ) $.  
In this article we will say that a function satisfies the \emph{Dirichlet condition} on a curve if it vanishes there and the \emph{Neumann condition} if 
its derivative along the normal to the curve vanishes. 

\begin{definition}[Eigenvalues for mixed Dirichlet and Neumann boundary conditions] 
\label{D:mixed} 
Suppose $\Lcalu$ and $\Uu$ are as in \ref{D:eigen} and moreover 
the boundary $\partial\Uu$ is piecewise-smooth and can be decomposed as $\partial \Uu= \partial_D\Uu \cup \partial_N\Uu$---note that 
$\partial_D\Uu$, $\partial_N\Uu$ can be empty.  
We define then the following for $i\in\N$ and $\lambda\in\R$: 
\\  
(i) 
$C_{pw}^\infty[\Uu;\partial_D\Uu,\partial_N\Uu]$ to be the space of piecewise-smooth functions on $\Uu$ which satisfy 
the Dirichlet condition on $\partial_D\Uu$ and the Neumann condition on $\partial_N\Uu$;  
\\  
(ii) 
$\lambda_i[\Lcalu,\Uu;\partial_D\Uu,\partial_N\Uu] := \lambda_i( \, \Lcalu,  C_{pw}^\infty[\Uu;\partial_D\Uu,\partial_N\Uu] \,)$;  
\\  
(iii) 
$\#_{<\lambda}[\Lcalu,\Uu;\partial_D\Uu,\partial_N\Uu] := \#_{<\lambda}( \, \Lcalu,  C_{pw}^\infty[\Uu;\partial_D\Uu,\partial_N\Uu] \,)$  
and similarly for ``$=\lambda$'' and ``$\le\lambda$''. 
\qed 
\end{definition}

\subsection*{Acknowledgments}
$\phantom{ab}$
\nopagebreak

The authors would like to thank Richard Schoen for his continuous support and interest in the results of this article 
and Otis Chodosh for bringing this problem to their attention. 
NK was partially supported by NSF grant DMS-1405537.  
NK would like to thank also for their hospitality and support the Institute for Advanced Study at Princeton during the 2018 Fall term,  
and the University of California, Irvine, during Spring 2019. 
Finally we would like to thank Robert Kusner and Richard Schoen for pointing out that our theorem 
implies isolatedness for the surfaces involved. 

\section{Basic spherical geometry}

\subsection*{Rotations along or about great circles} 
$\phantom{ab}$
\nopagebreak

Note that by \ref{D:refl},  
$C^\perp$ is the great circle furthest from a given great circle $C$ in $\Sph^3$.  
(Note that the points of $C^\perp$ are at distance $\pi/2$ in $\Sph^3$ from $C$ and any point of $\Sph^3\setminus C^\perp$ 
is at distance $<\pi/2$ from $C$). 
Equivalently $C^\perp$ is the set
of poles of great hemispheres with equator $C$;
therefore $C$ and $C^\perp$ are linked. 
The group 
$\Gsym^{C \cup C^\perp } $
contains 
$\Gsym^{C } = \Gsym^{C^\perp } $ 
(which includes arbitrary rotation or reflection in the two circles)  
and includes also   
orthogonal transformations exchanging 
$C$ with $C^\perp$.  

\begin{definition}[Rotations $\rot_C^\phi$, $\rot^C_\phi$ and Killing fields $K_{C}$, $K^{C}$] 
\label{l:D:rot} 
Given a great circle $C\subset\Sph^3$, $\phi \in \R$,
and an orientation chosen on the totally orthogonal circle $C^\perp$,
we define the following: 
\newline
(i) 
the rotation about $C$ by angle $\phi$ 
is the element $\rot_C^\phi$ of $SO(4)$ preserving $C$ pointwise 
and rotating the totally orthogonal circle $C^\perp$ along itself by angle $\phi$ 
(in accordance with its chosen orientation);  
\newline
(ii) 
the Killing field $K_{C}$ on $\Sph^3$ 
and the normalized Killing field $\Ktilde_{C}$ on $\Sph^3 \setminus C $ 
are given by 
$\left.\phantom{\frac12}K_{C}\right|_p := \left. \frac{\partial}{\partial\phi} \right|_{\phi=0} \rot_{C}^\phi(p)$
$\forall p\in\Sph^3$ 
and 
$\left.\Ktilde_{C}\right|_p \, := \,  
\frac{ \left.K_{C}\right|_p }{ \left| \left.K_{C}\right|_p \right| } 
$ 
$\forall p\in\Sph^3\setminus C$.   

Assuming now an orientation chosen on $C$ we define the following: 
\newline
(iii) 
the rotation along $C$ by angle $\phi$ is $\rot^C_\phi := \rot_{C^\perp}^\phi$;  
\newline
(iv) 
the Killing field $K^{C}:=K_{C^\perp}$ on $\Sph^3$ 
and the normalized Killing field $\Ktilde^{C}:=\Ktilde_{C^\perp}$ on $\Sph^3\setminus C^\perp$.  
\qed 
\end{definition}

Note that $\rot^C_\phi = \rot_{C^\perp}^\phi$ resembles a translation along $C$,  
while in the vicinity of $C^\perp$ it is a rotation. 
Note also that $K_{C}$ is defined to be a rotational Killing field around $C$, 
vanishing on $C$ and equal to the unit velocity on ${C}^\perp$.  

\begin{lemma}[Orbits]  
\label{orbits} 
For $K^C$ as in \ref{l:D:rot}, the orbits of $K^C$ (that is its flowlines) are planar circles  
and $\forall\pp\in C $ each orbit intersects the closed hemisphere $C^\perp \cone \pp$ exactly once.  
Moreover the intersection (when nontrivial) is orthogonal. 
\end{lemma} 

\begin{proof} 
This is straightforward to check already in $\R^4$ with the 
hemisphere $C^\perp \cone \pp$ replaced by the half-three-plane containing $\pp$ and with boundary $\Span(C^\perp)$. 
By restricting then to $\Sph^3$ the result follows. 
\end{proof} 

This lemma allows us to define a projection which effectively identifies the space of orbits in discussion 
with a closed hemisphere: 

\begin{definition}[Projections by rotations]  
\label{Pi} 
For $C$ and $\pp$ as in \ref{orbits} 
we define the smooth map 
$\Pi^C_\pp :\Sph^3\to C^\perp \cone \pp$ by requiring $\Pi^C_\pp x$ to be the intersection of $C^\perp \cone \pp$ with the orbit of $K^C$ containing $x$,    
for any $x \in\Sph^3$.  
\end{definition}

\begin{definition}[Graphical sets]  
\label{graphical} 
A set $A\subset \Sph^3$ is called \emph{graphical with respect to $K^C$} (with $C$ as above) 
if each orbit of $K^C$ intersects $A$ at most once. 
If moreover $A$ is a submanifold and there are no orbits of $K^C$ which are tangent to $A$, 
then $A$ is called 
\emph{strongly graphical with respect to $K^C$}.   
\end{definition}

\subsection*{The geometry of totally orthogonal circles}
$\phantom{ab}$
\nopagebreak

We fix now some $C$ and $C^\perp$ as above, and orientations on both. 
(Of course, after choosing an orientation on $C$, 
choosing an orientation on $C^\perp$ is equivalent to choosing an orientation on $\Sph^3$.)  
We define $\forall \phi\in\R$ the points 
\begin{equation}
\label{points} 
\pp_{\phi}=\pp_{\phi}[C]:= \rot_{C^\perp}^{\phi}\,\pp_0\,\in \, C,
\qquad
\pp^{\phi}=\pp^{\phi}[C]:= \rot_{C}^{\phi}\,\pp^0\,\in \, C^{\perp},
\end{equation}
where $\pp_0,\pp^0$ are arbitrarily fixed points on $C$ and $C^{\perp}$ respectively. 
Note that we will routinely omit $[C]$ when understood from the context.
Using \ref{span} we further define $\forall \phi\in\R$ the great spheres 
\begin{equation}
\label{hemispheres} 
\Sigma^\phi = \Sigma^\phi [C] :=  \,\Sph( C , \pp^\phi ) , 
\qquad  
\Sigma_\phi = \Sigma_\phi [C] := \, \Sph( C^\perp , \pp_\phi),   
\end{equation}
and $\forall\phi,\phi'\in\R$ 
the great circles 
\begin{equation}
\label{circles}  
C_\phi^{\phi'} = C_\phi^{\phi'} [C] := \Sph( \, \pp_\phi, \pp^{\phi'} \, ).  
\end{equation}

\begin{definition}[Coordinates on $\R^4$]  
\label{D:coordinates} 
Given $C$ as above and points as in \ref{points}, 
we define coordinates $(x^1,x^2,x^3,x^4)$ on $\R^4\supset\Sph^3$ by requiring that 
$$
\pp_0=(1,0,0,0), \qquad \pp_{\pi/2}=(0,1,0,0), \qquad \pp^0=(0,0,1,0), \qquad \pp^{\pi/2}=(0,0,0,1). 
$$
\end{definition} 

\begin{lemma}[Basic geometry related to $C$ and $C^\perp$]  
\label{L:obs} 
The following hold $\forall\phi,\phi', \phi_1,\phi'_1,\phi_2,\phi'_2  \in\R$.  
\\
(i) $ \pp_{\phi+\pi} = - \pp_\phi $  and 
$ \pp^{\phi+\pi} = - \pp^\phi $.   
Similarly 
$ \Sigma_{\phi+\pi} = \Sigma_\phi $  and 
$ \Sigma^{\phi+\pi} = \Sigma^\phi $.   
\\
(ii) 
$C_\phi^{\phi'} \cap C = \{ \pp_\phi , \pp_{\phi+\pi} \} $ 
and 
$C_\phi^{\phi'} \cap C^\perp  = \{ \pp^{\phi'} , \pp^{\phi'+\pi} \} $ 
with orthogonal intersections. 
Moreover 
$ 
C_\phi^{\phi'} 
=       
\overline{ \pp_\phi  \pp^{\phi'} } \, \cup \,      
\overline{ \pp^{\phi'} \pp_{\phi+\pi} } \, \cup \,      
\overline{ \pp_{\phi+\pi} \pp^{\phi'+\pi} } \, \cup \,      
\overline{ \pp^{\phi'+\pi} \pp_\phi  } 
. 
$ 
\\ 
(iii) 
$C \cone \pp^\phi$  and $C^\perp \cone \pp_\phi$  
are closed great hemispheres 
with boundary $C$ and $C^\perp$ and poles $\pp^\phi$ and $\pp_\phi$ respectively. 
\\ 
(iv) 
$\Sigma_\phi = ( C^\perp \cone \pp_\phi ) \cup ( C^\perp \cone \pp_{\phi+\pi} )$  
and 
$\Sigma^\phi = ( C \cone \pp^\phi ) \cup ( C \cone \pp^{\phi+\pi} )$.  
\\ 
(v) 
$\Sigma^\phi \cap C^\perp = \{ \pp^\phi , \pp^{\phi+\pi} \}$ 
and 
$\Sigma_\phi \cap C = \{ \pp_\phi, \pp_{\phi+\pi}\}$ 
with orthogonal intersections. 
\\ 
(vi) 
$C_\phi^{\phi'} = \Sigma_\phi \cap \Sigma^{\phi'}$  
with orthogonal intersection. 
\\
(vii) 
$\left( { C_\phi^{\phi'} } \right)^\perp = C_{\phi+\pi/2}^{\phi'+\pi/2}$. 
\\ 
(viii)  
$\Sigma^\phi \cap \Sigma^{\phi'} = C$ unless $\phi=\phi' \pmod \pi$ 
in which case 
$\Sigma^\phi = \Sigma^{\phi'}$.   
Similarly 
$\Sigma_\phi \cap \Sigma_{\phi'} = C^\perp$ unless $\phi=\phi' \pmod \pi$ 
in which case 
$\Sigma_\phi = \Sigma_{\phi'}$.   
In both cases the intersection angle is $\phi'-\phi \pmod \pi$. 
\\
(ix) 
$C_{\phi_1}^{\phi_1'} \cap C_{\phi_2}^{\phi_2'} = \emptyset$ 
unless 
$\phi_1=\phi_2 \pmod \pi$ 
or
$\phi'_1=\phi'_2 \pmod \pi$. 
If both conditions hold then  
$C_{\phi_1}^{\phi_1'} = C_{\phi_2}^{\phi_2'}$.  
If only the first condition holds then 
$C_{\phi_1}^{\phi_1'} \cap C_{\phi_2}^{\phi_2'} = 
\{ \pp_{\phi_1} , \pp_{\phi_1+\pi} \}$ 
with intersection angle equal to $\phi'_2-\phi'_1 \pmod \pi$. 
If only the second condition holds then 
$C_{\phi_1}^{\phi_1'} \cap C_{\phi_2}^{\phi_2'} = 
\{ \pp^{\phi_2} , \pp^{\phi_2+\pi} \}$ 
with intersection angle equal to $\phi_2-\phi_1 \pmod \pi$. 
\end{lemma}

\begin{proof} 
It is straightforward to verify all these statements by using the coordinates defined in 
\ref{D:coordinates}.  
\end{proof} 

\begin{definition}[Symmetries of Killing fields]  
\label{DKsymm} 
We call a Killing field $K$ \emph{even (odd) under an isometry $\refl$} if it satisfies 
${\refl}_* \circ K = K \circ \refl$ 
(${\refl}_* \circ K = - K \circ \refl$).  
\qed 
\end{definition} 

\begin{lemma}[Some symmetries of Killing fields]  
\label{Ksymm} 
The following hold $\forall \phi,\phi'\in\R$. 
\\ 
(i) 
$K_C$ is 
odd under $\refl_{\Sigma^{\phi}}$ 
and 
$\refl_{ C_{ \phi }^{ \phi' } }$ 
and even under $\refl_{\Sigma_{\phi}}$.  
\\
(ii) 
$K_{C^\perp}$ is 
odd under $\refl_{\Sigma_{\phi}}$ 
and 
$\refl_{ C_{ \phi }^{ \phi' } }$  
and even under $\refl_{\Sigma^{\phi}}$.  
\\ 
(iii) 
$K_{ C_{ \phi }^{ \phi' } }$ is odd under 
$\refl_{\Sigma_\phi}$ and $\refl_{\Sigma^{\phi'}}$ 
and even under 
$\refl_{\Sigma_{\phi+\pi/2}}$ and $\refl_{\Sigma^{\phi'+\pi/2}}$.  
Moreover 
${\Sigma_{\phi+\pi/2}}$ and ${\Sigma^{\phi'+\pi/2}}$  
are preserved under the flow of 
$K_{ C_{ \phi }^{ \phi' } }$ and contain the fixed points $\pm\pp^{\phi'}\in {\Sigma_{\phi+\pi/2}}$ and $\pm\pp_{\phi}\in {\Sigma^{\phi'+\pi/2}}$ 
and the geodesic orbit 
${ C_{ \phi +\pi/2 }^{ \phi' +\pi/2 } } = {\Sigma_{\phi+\pi/2}}\cap {\Sigma^{\phi'+\pi/2}}$.   
\end{lemma} 

\begin{proof} 
For any great circle $C'$ we have that $K_{C'}$ is even (odd) with respect to a reflection $\refl$ 
if and only if $\refl(C'^\perp) = C'^\perp$ and $\refl$ respects (reverses) the orientation of $C'^\perp$.  
Applying this it is straightforward to confirm the lemma. 
\end{proof} 

\section{Tessellations of $\Sph^3$}

\subsection*{Lawson tessellations}   
$\phantom{ab}$
\nopagebreak

Our purpose is to study the Lawson surfaces $\xi_{m-1,1}$ \cite{Lawson},  
which have genus $g=m-1$ and can be viewed as desingularizations of $\Sigma^{\pi/4}\cup \Sigma^{-\pi/4}$, 
where $m\ge3$, $m\in\N$.   
With this goal it is helpful to introduce the notation 
\begin{equation}
\label{qpoints} 
\begin{aligned}
t_{i'} :=& {(2i'-1)\frac\pi{2m}} \in \R , 
\qquad & 
t^{j'} :=& {(2j'-1)\frac\pi{4}} \in \R ,  
\\
\pqq_{i'} :=& \pp_{ t_{i'} } \in C , 
\qquad & 
\pqq^{j'} :=& \pp^{ t^{j'} } \in C^\perp ,  
\end{aligned} 
\qquad 
\forall i',j'\in\frac12\Z. 
\end{equation}
Note that we have then $2m$ points $\pqq_{i}$ for $i\in\Z$ subdividing $C$ into $2m$ equal arcs of length $\pi/m$ each, 
and $4$ points $\pqq^j$ for $j\in\Z$ subdividing $C^\perp$ into $4$ arcs of length $\pi/2$ each. 
$\pqq_{i+\frac12}$ is the midpoint of $\overline{\, \pqq_i \pqq_{i+1} \,}$ for each $i\in\Z$ and 
$\pqq^{j+\frac12}$ is the midpoint of $\overline{\, \pqq^j \pqq^{j+1} \,}$ for each $j\in\Z$. 

We define now $\forall i,j\in\Z$ compact domains 
$\Om_i,\Om^j,\Om_i^j$ by 
\begin{equation}
\label{Om} 
\begin{gathered}
\Om_i := 
C^\perp \cone \overline{\, \pqq_i \pqq_{i+1} \, } , 
\qquad 
\Om^j := 
C \cone \overline{\, \pqq^{j} \pqq^{j+1} \, },  
\qquad 
\Om_i^j := \Om_i \cap \Om^j = 
\overline{\, \pqq_i \pqq_{i+1} \pqq^{j} \pqq^{j+1} \, } \, . 
\end{gathered}
\end{equation} 
Clearly we have then the decompositions with disjoint interiors 
\begin{equation} 
\label{Omdec}
\Sph^3=\bigcup_{i=0}^{2m-1} \Om_i  
=\bigcup_{j=0}^3 \Om^j  
=\bigcup_{i=0}^{2m-1} \bigcup_{j=0}^3 \Om_i^j. 
\end{equation} 
Note that 
\begin{equation} 
\label{Omrot} 
\Om_i= \rot_{C^\perp}^{(i-i')\frac\pi{m} } \, \Om_{i'}, 
\qquad 
\Om^j= \rot_{C}^{(j-j')\frac\pi{2} } \, \Om^{j'}, 
\qquad 
\Om_i^j=\rot_{C^\perp}^{(i-i')\frac\pi{m} }  \, \rot_{C}^{(j-j')\frac\pi{2} } \, \Om_{i'}^{j'} . 
\end{equation} 
Moreover we have 
\begin{equation} 
\label{Omboundary} 
\begin{gathered} 
\partial \Om_i = 
C^\perp \cone \{ \pqq_i, \pqq_{i+1} \} ,  
\\    
\partial \Om^j 
=    C \cone \{ \pqq^j, \pqq^{j+1} \} .  
\end{gathered} 
\end{equation} 

\begin{lemma}[Properties of $\Om_i^j$] 
\label{Om:p} 
$\forall i,j\in\Z$, $\Om_i^j$ is a spherical tetrahedron and satisfies the following. 
\\
(i) 
Its faces are the spherical triangles 
$\overline{ \, \pqq_{ i } \pqq^{ j } \pqq^{ j+1 } \, } $, 
$\overline{ \, \pqq_{ i+1 } \pqq^{ j } \pqq^{ j+1 } \, } $, 
$\overline{ \, \pqq_{ i } \pqq_{ i+1 } \pqq^{ j } \, } $, 
and 
$\overline{ \, \pqq_{ i } \pqq_{ i+1 } \pqq^{ j+1 } \, } $. 
\\ 
(ii) 
Its dihedral angles are all $\pi/2$ except for the one along 
$\overline{ \, \pqq^{ j } \pqq^{ j+1 } \, } $ 
which is $\pi/m$. 
\\ 
(iii)
It is bisected by the spherical triangles 
$\overline{ \, \pqq_{ i+\frac12 } \pqq^{ j } \pqq^{ j+1 } \, } $ 
and 
$\overline{ \, \pqq_{ i } \pqq_{ i+1 } \pqq^{ j+\frac12 } \, } $   
and its symmetries are given by ($\refl_{\Sph^3}$ is the identity map on $\Sph^3$) 
\begin{equation}
\label{Omsym} 
\Gsym^{\Om_i^j}= \{ \, \refl_{\Sph^3}, 
\refl_{\Sigma_{i\pi/m}} , 
\refl_{\Sigma^{j\pi/2}} , 
\refl_{C_{i\pi/m}^{j\pi/2} } \, \} 
\simeq \Z_2\times \Z_2 
. 
\end{equation}  
(iv) It is convex in the sense that $\overline{xy}\subset\Om_i^j$ $\forall x,y\in \Om_i^j$. 
\end{lemma} 

\begin{proof}
It is straightforward to check all these statements by using the definitions and for (iii) that $m>2$. 
\end{proof} 

\ref{Om:p}.iii  
motivates us to define $\forall i,j\in\Z$, 
by modifying \ref{Om}, 
compact domains  
$\Om_{i\pm},\Om^{j\pm},\Om_{i\pm}^{j\pm}$  by 
\begin{equation}
\label{Ompm} 
\begin{gathered} 
\Om_{i\pm} :=  \, 
C^\perp \cone \overline{\, \pqq_{i+\frac12 } \pqq_{i+\frac12\pm\frac12} \, } , 
\qquad
\Om^{j\pm} :=  \, 
C \cone \overline{\, \pqq^{j+\frac12 } \pqq^{j+\frac12\pm\frac12} \, },  
\\
\Om_{i\pm}^{j\pm} := \,  \Om_{i\pm} \cap \Om^{j\pm}= 
\overline{\, \pqq_{i+\frac12} \pqq_{i+\frac12\pm\frac12} \pqq^{j+\frac12} \pqq^{j+\frac12\pm\frac12} \, } \, , 
\\
\Om_{i}^{j\pm} := \,  \Om_{i} \cap \Om^{j\pm}= 
\overline{\, \pqq_i \pqq_{i+1} \pqq^{j+\frac12} \pqq^{j+\frac12\pm\frac12} \, } \, , 
\\
\Om_{i\pm}^{j} := \,  \Om_{i\pm} \cap \Om^{j}= 
\overline{\, \pqq_{i+\frac12} \pqq_{i+\frac12\pm\frac12} \pqq^{j} \pqq^{j+1} \, } \, . 
\end{gathered} 
\end{equation} 
We have then various decompositions with disjoint interiors, for example 
\begin{equation} 
\label{Ompmdec}
\Om_i^j= \Om_{i-}^{j-} \cup \Om_{i+}^{j-} \cup \Om_{i-}^{j+} \cup \Om_{i+}^{j+}, 
\qquad 
\Om_{i+}^j=\Om_{i+}^{j-} \cup  \Om_{i+}^{j+}.  
\end{equation}  
Note also that 
\begin{equation} 
\label{Omrefl} 
\Om_{i+}^{j-} =  \refl_{\Sigma^{j\pi/2}}  \Om_{i+}^{j+} , 
\qquad 
\Om_{i-}^{j+} =  \refl_{\Sigma_{i\pi/m}}  \Om_{i+}^{j+} , 
\qquad 
\Om_{i-}^{j-} =  \refl_{C_{i\pi/m}^{j\pi/2} }  \Om_{i+}^{j+} .
\end{equation} 
Moreover 
all four tetrahedra $\Om_{i\pm}^{j\pm}$ have  
$\overline{ \pp_{ i\frac\pi m }  \pp^{ j\frac\pi 2 } } = \overline{ \pqq_{i+\frac12} \pqq^{ j+\frac12 } } $ 
as a common edge and adjacent ones have 
common faces given by 
$\Om_{i-}^{j+} \cap \Om_{i-}^{j-} = \overline{\, \pqq_{i+\frac12} \pqq_{i} \pqq^{j+\frac12} \, } $, 
$\Om_{i+}^{j+} \cap \Om_{i+}^{j-} = \overline{\, \pqq_{i+\frac12} \pqq_{i+1} \pqq^{j+\frac12} \, } $, 
$\Om_{i-}^{j-} \cap \Om_{i+}^{j-} = \overline{\, \pqq_{i+\frac12} \pqq^{j+\frac12} \pqq^{j} \, } $, 
and 
$\Om_{i-}^{j+} \cap \Om_{i+}^{j+} = \overline{\, \pqq_{i+\frac12} \pqq^{j+\frac12} \pqq^{j+1} \, } $. 

\subsection*{Subdividing $\Sph^3$ with mutually orthogonal two-spheres} 
$\phantom{ab}$
\nopagebreak

Note that by \ref{L:obs}.vi,viii   
$\Sigma^0$, $\Sigma^{\pi/2}$, $\Sigma_0$, and $\Sigma_{\pi/2}$ 
form a system of four mutually orthogonal two-spheres in $\Sph^3$. 
We will later study the subdivisions these two-spheres effect on $\Sph^3$ and the Lawson surfaces.    
To this end we define 
$\Omu_{**}^{\pm*}$, 
$\Omu_{**}^{*\pm}$, 
$\Omu^{**}_{\pm*}$, 
and 
$\Omu^{**}_{*\pm}$, 
to be the closures of the connected components into which $\Sph^3$ is subdivided by the removal of 
$\Sigma^0$, $\Sigma^{\pi/2}$, $\Sigma_0$, or $\Sigma_{\pi/2}$ respectively, 
chosen so that 
\begin{equation} 
\label{Omu+} 
\pp^{\pm\pi/2}\in \Omu_{**}^{\pm*}, \quad  
\pp^{\frac{\pi}2\mp\frac{\pi}2}\in \Omu_{**}^{*\pm}, \quad  
\pp_{\pm\pi/2}\in \Omu^{**}_{\pm*}, \quad  
\pp_{\frac{\pi}2\mp\frac{\pi}2}\in \Omu^{**}_{*\pm}.  
\end{equation} 
To further subdivide we replace $*$'s by $\pm$ signs to denote the corresponding intersections 
of these domains; for example we have  
\begin{equation} 
\label{Omu+-} 
\Omu_{+-}^{-*} := \Omu_{+*}^{**} \cap \Omu_{*-}^{**} \cap \Omu_{**}^{-*}. 
\end{equation}  

Clearly we have 
\begin{equation} 
\label{partialOmu+} 
\partial \Omu_{**}^{\pm*} = \Sigma^0 , \quad  
\partial \Omu_{**}^{*\pm} = \Sigma^{\pi/2} , \quad  
\partial \Omu^{**}_{\pm*} = \Sigma_0 , \quad  
\partial \Omu^{**}_{*\pm} = \Sigma_{\pi/2} .  
\end{equation} 

\begin{lemma}[Elementary geometry of $\Omu^{++}_{++}$]  
\label{Omu:p} 
$\Omu^{++}_{++}$ is the spherical tetrahedron 
$\overline{\pp^0 \pp^{\pi/2} \pp_0 \pp_{\pi/2} }$ 
and satisfies the following. 
\\
(i) 
Its faces are the spherical triangles 
$\overline{\pp_{\pi/2} \pp^{\pi/2} \pp^0} \subset \Sigma_{\pi/2}$, 
$\overline{\pp_{\pi/2} \pp^{\pi/2} \pp_0} \subset \Sigma^{\pi/2}$, 
$\overline{\pp_{0} \pp^{\pi/2} \pp^0} \subset \Sigma_0$, 
$\overline{\pp_{0} \pp_{\pi/2} \pp^0}  \subset \Sigma^0$. 
All angles of all faces are $\pi/2$. 
\\ 
(ii)  
All its edges have length $\pi/2$ and its dihedral angles are all $\pi/2$. 
\\ 
(iii) 
Its symmetry group is isomorphic to the symmetric group on its vertices. 
$\Omu^{++}_{++}$ is bisected by six spherical triangles including 
$\overline{ \, \pp_{ \pi/4 } \pp^{ 0 } \pp^{ \pi/2 } \, } $ 
and 
$\overline{ \, \pp_{ 0 } \pp_{ \pi/2 } \pp^{ \pi/4 } \, } $   
and its symmetries include 
$\refl_{\Sigma_{\pi/4}}$, 
$\refl_{\Sigma^{\pi/4}}$, 
and 
$\refl_{C_{\pi/4}^{\pi/4} }$. 
\end{lemma}

\begin{proof}
It is straightforward to verify all these statements by using the definitions. 
\end{proof}

\begin{lemma}[Some decompositions] 
\label{OmuOm} 
We have the following. 
\\ 
(i)  
$\Omu^{++}_{**} = \cup_{i=0}^{2m-1} ( \Om^{0+}_{i} \cup \Om^{1-}_{i} )$.  
\\ 
(ii)  
$\Omu^{++}_{+*} = 
\Omu^{++}_{++} \cup \refl_{\Sigma_{\pi/2}} \Omu^{++}_{++} = 
\Om^{0+}_{0+} \cup \Om^{1-}_{0+}  \cup  
\left( \cup_{i=1}^{{m}-1} ( \Om^{0+}_{i} \cup \Om^{1-}_{i} )  \right) 
\cup \Om^{0+}_{{m} - } \cup \Om^{1-}_{{m} - }  
$.  
\\ 
(iii)  
$\Omu^{++}_{++} = \Om^{0+}_{0+} \cup \Om^{1-}_{0+}  \cup  
\left\{ 
\begin{aligned}
&  \left( \cup_{i=1}^{\frac{m}2-1} ( \Om^{0+}_{i} \cup \Om^{1-}_{i} )  \right) \cup \Om^{0+}_{\frac{m}2 - } \cup \Om^{1-}_{\frac{m}2 - },                
&& \text{ if } m\in 2\Z,  
\\
& \left( \cup_{i=1}^{\frac{m-1}2} ( \Om^{0+}_{i} \cup \Om^{1-}_{i} ) \right),   && \text{ if } m\in 2\Z+1.  
\end{aligned} 
\right. 
$ 
\end{lemma}

\begin{proof}
It is straightforward to verify all these statements by using the definitions. 
\end{proof}

\subsection*{The coordinate Killing fields}
$\phantom{ab}$
\nopagebreak

Using the coordinates defined in \ref{D:coordinates}, 
we endow $\R^4$ with its standard orientation $dx^1 \wedge dx^2 \wedge dx^3 \wedge dx^4$,
and we endow the six coordinate $2$-planes with the orientations
  \begin{equation}
  \begin{aligned}
  	&dx^1 \wedge dx^2, \quad dx^3 \wedge dx^4, \quad dx^1 \wedge dx^4, \\
  	&dx^2 \wedge dx^3, \quad dx^1 \wedge dx^3, \quad \mbox{and} \quad dx^4 \wedge dx^2.
  \end{aligned}
  \end{equation}
Note that these orientations have been chosen so that one obtains the orientation
of $\R^4$ upon taking the wedge product (in either order)
of the orientation forms of a pair of orthogonally complementary $2$-planes.

In turn we orient each coordinate unit circle by taking the interior product
of its outward unit normal with the orientation form of the $2$-plane it spans.
These choices are consistent with the convention that for any oriented great circle $C'$
we orient ${C'}^\perp$ so that the wedge product of
the two corresponding $2$-plane orientations will yield the standard orientation on $\R^4$.
Thus
\begin{equation}
\label{E:killing} 
    \begin{aligned}
      K_{C^\perp}(x)=K^{C}(x)&=x^1 \, \pp_{\pi/2} - x^2 \, \pp_0 \\ 
      K_C(x)=K^{C^\perp}(x)&=x^3 \, \pp^{\pi/2} - x^4 \, \pp^0 , \\
      K_{C_{\pi/2}^{\pi/2}}(x)=K^{C_0^0}(x)&=x^1 \, \pp^0 - x^3 \, \pp_0 , \\
      K_{C_0^0}(x)=K^{C_{\pi/2}^{\pi/2}}(x)&=x^4 \, \pp_{\pi/2} - x^2 \, \pp^{\pi/2} , \\
      K_{C_0^{\pi/2}}(x)=K^{C_{\pi/2}^0}(x)&=x^2 \, \pp^0 - x^3 \, \pp_{\pi/2} , \mbox{ and} \\
      K_{C_{\pi/2}^0}(x)=K^{C_0^{\pi/2}}(x)&=x^1 \, \pp^{\pi/2} - x^4 \, \pp_0 .
    \end{aligned}
   \end{equation}

\begin{lemma}[$K_{C_{\phi}^{\phi'}}$ on $\Omu^{++}_{+\pm}$ for $\phi,\phi'\in\{0,\pi/2\}$]  
\label{Omu:K} 
We have the following (recall \ref{l:D:rot} and \ref{E:killing}).  
\\
(i)  
$\rot_{C_0^0} \pp^{\pi/2} = \pp_{\pi/2} $, 
$\Ktilde_{C_0^0} ( \Omu^{++}_{+*} ) = \overline{ \, \pp_{\pi/2} \pp^{-\pi/2} \,}$.  
\\
(ii)  
$\rot_{C_0^{\pi/2}} \pp_{\pi/2} = \pp^{0} $, 
$\Ktilde_{C_0^{\pi/2}} ( \Omu^{++}_{+*} ) = \overline{ \, \pp^{0} \pp_{-\pi/2} \,}$.  
\\
(iii)  
$\rot_{C^0_{\pi/2}} \pp_{0} = \pp^{\pi/2} $, 
$\rot_{C^0_{\pi/2}} \pp^{\pi/2} = \pp_{\pi} $, 
$\Ktilde_{C^0_{\pi/2}} ( \Omu^{++}_{++} ) = \overline{ \, \pp^{\pi/2} \pp_{\pi} \,}$,  
$\Ktilde_{C^0_{\pi/2}} ( \Omu^{++}_{+-} ) = \overline{ \, \pp_{\pi} \pp^{-\pi/2} \,}$.  
\\
(iv) 
$\rot_{C^{\pi/2}_{\pi/2}} \pp_{0} = \pp^{0} $, 
$\rot_{C^{\pi/2}_{\pi/2}} \pp^{0} = \pp_{\pi} $, 
$\Ktilde_{C^{\pi/2}_{\pi/2}} ( \Omu^{++}_{++} ) = \overline{ \, \pp^{0} \pp_{\pi} \,}$,  
$\Ktilde_{C^{\pi/2}_{\pi/2}} ( \Omu^{++}_{+-} ) = \overline{ \, \pp_{\pi} \pp^{\pi} \,}$.  
\end{lemma} 

\begin{proof}
All claims follow easily from \eqref{E:killing} and Definition \ref{l:D:rot}.  
\end{proof}

\subsection*{Some quadrilaterals in $\Sph^3$} 
$\phantom{ab}$
\nopagebreak

We consider now $\forall i,j \in \Z$ 
the spherical quadrilateral $Q_i^j\subset \partial \Om_i^j$  
consisting of the four edges of the spherical tetrahedron $\Om_i^j$ 
not contained in $C$ or $C^\perp$; 
that is 
\begin{equation} 
\label{D:Q} 
Q_i^j :=
\overline{ \pqq_{ i } \pqq^{ j  } }  
\cup 
\overline{ \pqq^{ j } \pqq_{ i+1  } }  
\cup 
\overline{ \pqq_{ i+1 } \pqq^{ j+1  } }  
\cup 
\overline{ \pqq^{ j+1 } \pqq_{ i  } }  
. 
\end{equation} 
For ease of reference we define the set of vertices of $Q_i^j$ 
(the same as the set of vertices of $\Om_i^j$) 
\begin{equation} 
\label{D:Qv} 
\Qv_i^j :=
\{ 
\pqq_{ i } , 
\pqq_{ i+1  } ,  
\pqq^{ j  } ,  
\pqq^{ j+1 } 
\} 
. 
\end{equation}

Recall that by \ref{Omsym} $\forall i,j\in\Z$ the circle  
$\Sph( \pqq_{i+\frac12},\pqq^{j+\frac12} ) = {C_{i\pi/m}^{j\pi/2} } $ is an axis of symmetry of $\Om_i^j$. 
It is natural then to call this circle the ``axis'' of $\Om_i^j$ 
and study rotations along it as in the following lemma.  
We also define 
\begin{equation} 
\label{E:partialpm} 
\partial_+ \Om_i^j := \overline{ \, \pqq_{ i } \pqq^{ j } \pqq^{ j+1 } \, } \cup \overline{ \, \pqq_{ i+1 } \pqq^{ j } \pqq^{ j+1 } \, } 
\qquad \text{ and } \qquad 
\partial_- \Om_i^j := \overline{ \, \pqq_{ i } \pqq_{ i+1 } \pqq^{ j } \, } \cup \overline{ \, \pqq_{ i } \pqq_{ i+1 } \pqq^{ j+1 } \, },   
\end{equation} 
and by \ref{Om:p} we have then 
\begin{equation} 
\label{E:partialpm2} 
\partial\Om_i^j= \partial_+ \Om_i^j \cup \partial_- \Om_i^j 
\qquad \text{ and } \qquad 
Q_i^j= \partial_+ \Om_i^j \cap \partial_- \Om_i^j.  
\end{equation} 

\begin{lemma}[{$\Om_i^j$} and rotations along its axis] 
\label{L-alex} 
The following are true $\forall i,j\in\Z$ and any orbit $O$ of 
$K_{\widetilde{C}}$, 
where 
$\widetilde{C} := ({C_{i\pi/m}^{j\pi/2} })^\perp =  C_{i\pi/m+\pi/2}^{j\pi/2+\pi/2} $. 
\\ (i) $( \rot^t_{\widetilde{C}} \Om_i^j ) \cap \Om_i^j= \emptyset$ for $t\in (-3\pi/2, -\pi/2)\cup  (\pi/2,3\pi/2)$. 
Moreover either 
$( \rot^{\pm\pi/2}_{\widetilde{C}} \Om_i^j ) \cap \Om_i^j   =\{ \pqq_{i+\frac12} \}$ or  
$( \rot^{\pm\pi/2}_{\widetilde{C}} \Om_i^j ) \cap \Om_i^j   =\{ \pqq^{j+\frac12} \}$ 
(depending or the orientation of ${C_{i\pi/m}^{j\pi/2} } $ and the sign).  
\\ (ii) 
For each $\Om_{i\pm}^{j\pm}$ either $O \cap \Om_{i\pm}^{j\pm} = O \cap \Om_{i}^{j} $  or $O \cap \Om_{i\pm}^{j\pm} = \emptyset$.  
\\ (iii) 
If $O\cap \Om_i^j \ne \emptyset$,  
then (recall \ref{E:partialpm})   
$O\cap \partial_{\pm} \Om_i^j = \{x_{\pm}\} $ for some $x_{\pm} \in \partial_{\pm} \Omega_i^j$.  
Moreover $O\cap\Om_i^j$ is 
a connected arc (possibly a single point) whose endpoints are $x_+$ and $x_-$.  
\\ (iv) 
If $O\cap Q_i^j\ne \emptyset$, then $x_+=x_-\in Q_i^j$  
and 
$O\cap\Om_i^j=\{x_+\}$.   
\\ (v) 
If $O\cap ( \Om_i^j \setminus \Qv_i^j ) \ne \emptyset$,  
then $O$ intersects each face of $\Om_i^j$ containing $x_+$ ($x_-$) transversely. 
\\ (vi) 
$\Pi_i^j(\Om_i^j) \subset C_{i\pi/m+\pi/2}^{j\pi/2+\pi/2} \cone \pp_{i\pi/m} $ 
is homeomorphic to a closed disc with boundary 
$\Pi_i^j(Q_i^j) $, 
where 
$\Pi_i^j:= \Pi^{C_{i\pi/m}^{j\pi/2}}_{\pp_{i\pi/m} } $  
is defined as in \ref{Pi}  
(recall also $\pp_{i\pi/m}  = \pqq_{i+\frac12}$).  
\end{lemma}

\begin{proof} 
We can clearly assume without loss of generality that $i=j=0$. 
To prove (i) note that $H:=\pp^0\cone C^{\pi/2}_{\pi/2}$ and $H':=\pp_0\cone C^{\pi/2}_{\pi/2}$ are orthogonal 
closed hemispheres with common boundary $C^{\pi/2}_{\pi/2}$, 
intersecting $C^0_0$ orthogonally at $\pp^0$ and $\pp_0$ respectively, 
and satisfying  
$\rot^{\pi/2}_{\widetilde{C}} (H')=H$. 
Moreover $\overline{\pqq^0\pqq^1}\subset H$ and $\overline{\pqq_0\pqq_1}\subset H'$ with both geodesic segments avoiding the boundary $C^{\pi/2}_{\pi/2}$. 
Since two orthogonal hyperplanes separate $\R^4$ into four convex connected components,
(i) follows easily.
Because each of the bisecting spheres $\Sigma_0$ and $\Sigma^0$
is preserved by the family $\rot^{C_0^0}_t$,
the orbits of $K^{C_0^0}$ cannot cross either sphere, proving (ii).

Before turning to the remaining items
we first show that no orbit of $K^{C_0^0}$ intersects any face of $\Omega_0^0$
tangentially, except at a vertex.
By the symmetries it suffices to prove that orbits
intersect $\overline{\pp^0\pqq_1\pqq^1} \subset \Sigma_{\pi/2m}$ 
and $\overline{\pp_0\pqq_1\pqq^1} \subset \Sigma^{\pi/4}$ transversely (if at all)
except at $\pqq_1$ (the orbit through which is tangential to $\Sigma_{\pi/2m}$)
and $\pqq^1$ (the orbit through which is tangential to $\Sigma^{\pi/4}$).
Of course the spheres $\Sigma_{\pi/2m}$ and $\Sigma^{\pi/4}$ are minimal surfaces
and neither contains $C_0^0$,
so the Killing field $K^{C_0^0}$
induces a nontrivial Jacobi field on each of them.
A point where an orbit meets one of these spheres tangentially is a zero
of the corresponding Jacobi field,
but we know these nontrivial Jacobi fields are simply first harmonics,
each of whose nodal sets consists of a single great circle.
Clearly the reflection $\refl_{\Sigma^{\pi/2}}$ ($\refl_{\Sigma_{\pi/2}}$) preserves
the sides of $\Sigma_{\pi/2m}$ ($\Sigma^{\pi/4}$)
and reverses each orbit of $K^{C_0^0}$.
Thus orbits can meet $\Sigma_{\pi/2m}$ ($\Sigma^{\pi/4}$) tangentially
only along $C_{\pi/2m}^{\pi/2}$ ($C_{\pi/2}^{\pi/4}$),
which intersects $\overline{\pqq_1\pqq^1}$ only at $\pqq_1$ ($\pqq^1$),
establishing the asserted transversality.

Next we argue that no orbit of $K^{C_0^0}$ intersects
any face of $\Omega_0^0$ at more than one point.
Again (by the symmetries) it suffices to
show that every orbit intersects each of the faces
$\overline{\pp^0\pqq_1\pqq^1} \subset \Sigma_{\pi/2m}$ 
and $\overline{\pp_0\pqq_1\pqq^1} \subset \Sigma^{\pi/4}$
at most once.
To see this first note that the orbits of $K^{C_0^0}$ in $\R^4 \supset \Sph^3$
are planar circles, so if one intersects a great $2$-sphere at more than one point,
then the intersection must be either a great circle (the entire orbit) or a pair of points.
In the first case the $2$-sphere so intersected must contain $C_0^0$,
but neither the sphere $\Sigma^{\pi/4} \supset \overline{\pp_0\pqq_1\pqq^1}$
nor the sphere $\Sigma_{\pi/2m} \supset \overline{\pp^0\pqq_1\pqq^1}$
contains $C_0^0$, and so the orbits of $K^{C_0^0}$ must meet these spheres at most twice.
However, the reflection $\refl_{\Sigma^{\pi/2}}$ ($\refl_{\Sigma_{\pi/2}}$)
preserves both $\Sigma_0$ ($\Sigma^{\pi/4}$)
and each orbit (as a set) of $K^{C_0^0}$,
so that if an orbit intersects $\Sigma_0$ ($\Sigma^{\pi/4}$) in two points,
these points must lie on opposite sides of $\Sigma^{\pi/2}$ ($\Sigma_{\pi/2}$).
Since in fact $\Omega_0^0$ crosses neither sphere of symmetry,
we see that any orbit meets each face at most once, as claimed.

Now we are ready to prove (iii), (iv), and (v).
By the symmetries
it suffices to consider an orbit $O$ intersecting $\Omega_{0+}^{0+}$.
Of course by (i) $O$ is not contained in $\Omega_{0+}^{0+}$
and obviously by (ii) $O$ can enter (or exit) $\Omega_{0+}^{0+}$ only through
$\overline{\pp_0\pqq_1\pqq^1}$ or $\overline{\pp^0\pqq_1\pqq^1}$,
but by the preceding paragraph it intersects each at most once.
Since $\overline{\pqq_1 \pqq^1}$ lies on both these triangles,
it follows that any orbit $O$ meeting $\overline{\pqq_1 \pqq^1}$
intersects $\Omega_0^0$ at only one point.
If on the other hand $O$ misses $\overline{\pqq_1 \pqq^1}$,
then, by the transversality above,
it must intersect the interior of $\Omega_0^0$,
so in this case it must cross $\overline{\pp_0\pqq_1\pqq^1} \cup \overline{\pp^0\pqq_1\pqq^1}$ at least twice,
meaning, by the above, that in fact $O$ must intersect each of these triangles exactly once.
This completes the proof of (iii), (iv), and (v).
 
For (vi) set $\Pi:=\Pi_0^0$.
Since the quadrilateral $Q_0^0$ is itself a closed curve missing
$C_{\pi/2}^{\pi/2}=\Pi^{-1}\left(C_{\pi/2}^{\pi/2}\right)$,
its image $Q':=\Pi\left(Q_0^0\right)$ under $\Pi$
is likewise a closed curve missing $C_{\pi/2}^{\pi/2}$.
By item (iv) (and the embeddedness of $Q_0^0$)
it follows that $Q'$ is an embedded closed curve in the interior of $C_{\pi/2}^{\pi/2} \cone \pp_0$,
so that $\left(C_{\pi/2}^{\pi/2} \cone \pp_0\right) \backslash Q'$
has two connected components, one the disc bounded by $Q'$
and the other the annulus bounded by $Q'$ and $C_{\pi/2}^{\pi/2}$.
Call the closure of the disc $D'$.
Since the hemisphere
$\Pi^{-1}\left(\overline{\pp_0 \pp_{\pi/2}}\right)=C_0^0 \cone \pp_{\pi/2} \subset \Sigma^0$
intersects $Q_0^0$ only at $\pqq_1$,
we see that the geodesic arc $\overline{\pp_0 \pp_{\pi/2}}$
intersects $Q'$ exactly once (at $\pqq_1$),
and so we conclude that $\pp_0 \in D'$.
A second application of item (iv)
ensures that $\Pi\left(\Omega_0^0 \backslash Q_0^0\right)$
misses $Q'$,
but $\Omega_0^0 \backslash Q_0^0$ is connected and includes $\pp_0$,
so we have $\Pi\left(\Omega_0^0\right) \subset D'$.
Last, note that $D'':=\overline{\pqq_0\pqq_1\pqq^0} \cup \overline{\pqq_0\pqq_1\pqq^1}$
is a disc in $\Omega_0^0$ whose boundary is $Q_0^0$ and thereby mapped by $\Pi$
homeomorphically onto $Q'=\partial D'$.
It follows (by degree theory) that $\Pi(D'')=D'$,
and so of course $\Pi(\Omega_0^0)=D'$ as well.
\end{proof}

\section{The Lawson surfaces}

\subsection*{Definition, uniqueness, and symmetries} 
$\phantom{ab}$
\nopagebreak

Note that the surfaces we define below are the surfaces called $\xi_{m-1,1}$ in \cite{Lawson}. 
Recall that these surfaces can be viewed as desingularizations of two orthogonal great two-spheres. 
In this article we do not consider any other Lawson surfaces and when we refer to Lawson surfaces we mean these surfaces only. 
The surfaces defined in the next theorem are positioned so that they can be viewed as desingularizations of $\Sigma^{\pi/4}\cup \Sigma^{ - \pi/4}$ along $C$. 
Note also that we restrict our attention to the case $m\ge3$ because the surfaces produced 
otherwise are the great sphere ($m=1$) and the Clifford torus ($m=2$). 

\begin{theorem}[Lawson 1970 \cite{Lawson}] 
\label{T:lawson} 
Given an integer $m\ge3$ 
there is a unique compact connected 
minimal surface $D_i^j \subset \Om_i^j$ with 
$\partial D_i^j = Q_i^j$ (recall \eqref{Om} and \eqref{D:Q}).  
Moreover $D_i^j$ is a disc, minimizing area among such discs, 
and  
$$
M=M[C,m] :=\bigcup_{i+j\in2\Z} D_i^j
$$ 
is an embedded connected closed (so two-sided) smooth minimal surface of genus $m-1$. 
\end{theorem} 

\begin{proof} 
The theorem except for the uniqueness part but including 
the existence of a minimizing disc $D_i^j$ is proved in \cite{Lawson}. 
Although the uniqueness is also claimed in \cite{Lawson}, 
the subsequent literature 
(for example \cite{choe:soret:2009}) does not assume uniqueness known. 
We provide now a simple proof of uniqueness. 

Suppose ${D'}_i^j$ is another connected minimal surface in $\Om_i^j$ with boundary $Q_i^j$. 
By \ref{L-alex} $\rot^t_{C_{i\pi/m+\pi/2}^{j\pi/2+\pi/2} } D_i^j$ cannot intersect 
${D'}_i^j$ for any $t\in(-\pi,0)\cup(0,\pi)$ because 
otherwise we can consider the $\sup$ or $\inf$ of such $t$'s which we call $t'$. 
For $t'$ then we would have tangential contact on one side in the interior. 
By the maximum principle 
\cite{schoen1983}*{Lemma 1}  
this would imply 
equality of the surfaces and the boundaries, a contradiction. 

By \ref{L-alex} 
the orbits which are close enough to $Q_i^j\setminus\Qv_i^j$ and intersect $\Om_i^j$  
also intersect $D_i^j$ and ${D'}_i^j$. 
Since there are no intersections for $t\ne0$ above, we conclude that  
$D_i^j$ and ${D'}_i^j$ agree on a neighborhood of $Q_i^j\setminus\Qv_i^j$ and therefore by analytic continuation they are identical. 
\end{proof} 

\begin{corollary}[Symmetries of the Lawson discs] 
\label{Dsym0}
$\forall i,j\in\Z$ $D_i^j$ inherits the symmetries of $\Om_i^j$: it is preserved as a set by   
$\refl_{\Sigma_{i\pi/m}} = \refl_{\Sigma_{t_{i+\frac12}} } =\refl_{C^\perp,\pqq_{i+\frac12} }$, 
$\refl_{\Sigma^{j\pi/m}} = \refl_{\Sigma^{t^{j+\frac12}} } =\refl_{C,\pqq^{j+\frac12} }$, 
and 
the composition of those 
$\refl_{C_{i\pi/m}^{j\pi/2} }$.  
Moreover it has no more symmetries. 
\end{corollary} 

\begin{proof} 
That the symmetries of $\Om_i^j$ are symmetries of $D_i^j$ follows from the uniqueness of $D_i^j$ discussed in \ref{T:lawson}. 
Any symmetry of $D_i^j$ has to be a symmetry of its boundary and then of its vertices, and hence of $\Om_i^j$ as well. 
By \ref{Om:p}.iii this completes the proof. 
\end{proof}

\begin{lemma}[Generating symmetries of the Lawson surfaces]
\label{Msym}
For $M=M[C,m]$ as in \ref{T:lawson} we have the following symmetries, 
which generate $\Gsym^M$. 
\\
(i)  
$\forall i,j\in\Z$ we have  
$\refl_{\Sigma^{j\pi/2}}, \refl_{\Sigma_{i\pi/m}} \in \Gsym^M$. 
Moreover  
the collection of the great two-spheres of symmetry of $M$ is 
$\{\Sigma^{ j\pi /2 }\}_{j\in\Z} 
\cup 
\{\Sigma_{ i\pi /m }\}_{i\in\Z} $  and contains $m+2$ spheres. 
\\
(ii)  
$\forall i,j\in\Z$ we have  
$\refl_{C_{ (2i- 1)\frac\pi{2m} }^{ (2j- 1)\frac\pi4 } } 
= \refl_{\pqq_{i}, \pqq^{j}  } 
\in \Gsym^M$. 
Moreover 
the collection of great circles contained in $M$ is  
$\left\{ C_{ (2i- 1)\frac\pi{2m} }^{ (2j- 1)\frac\pi4  \,} 
= \Sph (\pqq_{i}, \pqq^{j}  ) 
\right\}_{i,j\in\Z}$  
and contains $2m$ great circles.  

Furthermore if $\nu:M\to\Sph^3$ is a unit normal smoothly chosen on $M$, then 
$\nu$ is even under the symmetries in (i) 
(that is for such a symmetry $\refl$ we have $\nu\circ\refl = \refl_* \circ \nu$) 
and odd under the symmetries in (ii) 
(that is for such a symmetry $\refl$ we have $\nu\circ\refl = - \refl_* \circ \nu$). 
\end{lemma} 

\begin{proof} 
Set $\mathbf{Q}:=\left\{Q_i^j\right\}_{i+j \in 2\Z}$
and $\mathbf{\Omega}:=\left\{\Omega_i^j\right\}_{i+j \in 2\Z}$.
It is easy to see (keeping in mind that $m>2$)
that an element of $O(4)$ permutes $\mathbf{Q}$
if and only if it permutes $\mathbf{\Omega}$.
By the uniqueness assertion of Theorem \ref{T:lawson}
any element of $O(4)$ permuting $\mathbf{Q}$
is then a symmetry of $M$.
Conversely,
since $M$ is disjoint from the interior of every $\Omega_i^j$ with $i+j \in 2\Z+1$,
every element of $\Grp_{sym}^M$
must permute $\mathbf{Q}$.
Now write $\Grp$ for the subgroup of $O(4)$
generated by all the orthogonal transformations named in the statement of the lemma.
It is immediately verified from definitions 
\ref{D:refl}, \ref{l:D:rot}, \ref{hemispheres}, \ref{circles}, \ref{Om}, and \ref{D:Q}
that every element of $\Grp$ indeed permutes
$\mathbf{Q}$,
confirming that $\Grp \subseteq \Grp_{sym}^M$.
In fact it is clear that $\Grp$ acts transitively on $\mathbf{\Omega}$,
so in order to show that $\Grp_{sym}^M \subseteq \Grp$
it suffices to show that any orthogonal transformation
preserving $\Omega_0^0$ as a set belongs to $\Grp$,
but this is evident from \ref{Omsym}.  
Thus $\Grp_{sym}^M=\Grp$.

The counts of the spheres and circles named in (i) and (ii)
are obvious from \eqref{hemispheres} and \eqref{circles} alone.
It is also obvious from the definitions that every circle
in item (ii) of the lemma
indeed lies on $M$,
and furthermore for this very reason
reflection through such a circle
must reverse $\nu$.
Note that for each $j \in \Z$ the symmetry $\refl_{\Sigma^{j\pi/2}}$
fixes $C$ pointwise.
In particular $\refl_{\Sigma^{j\pi/2}}$
fixes the point $\pqq_1 \in C \cap M$,
but $\pqq_1$
lies on each of the circles of symmetry
$C_{\pi/2m}^{\pi/4}$ and $C_{\pi/2m}^{-\pi/4}$
orthogonally intersecting $C$ there,
and so $\nu(\pqq_1)$ points along $C$
and is thereby preserved by $\refl_{\Sigma^{j\pi/2}}$.
A similar argument shows that $\nu$ is likewise preserved by every $\refl_{\Sigma_{i\pi/m}}$
with $i \in \Z$.

The only assertions left to prove are
that $M$ is invariant under no spheres of symmetry
other than those enumerated in (i)
and that $M$ contains no circles of symmetry other than those enumerated in (ii)
(since the reflection principle \cite{Lawson}*{Proposition 3.1}  
then ensures that $M$ contains no other great circles at all).
Accordingly suppose that $S$ is such a sphere or circle of symmetry.
so that $\refl_S \in \Grp_{sym}^M$.
As explained above,
$\refl_S$ therefore permutes $\mathbf{Q}$,
but because $m>2$, this requires in particular that $\refl_S$ preserve $C$ (and so $C^\perp$ too) as a set.
If $S$ is a great sphere, it must consequently intersect either $C$ or $C^\perp$ orthogonally
(containing the other), but to permute $\mathbf{Q}$ it can then be only one of the spheres listed in (i)
(a sphere bisecting some $\Omega_i^j$, since reflection through a sphere containing a face of an $\Omega_i^j$
takes the corresponding $Q_i^j$ to a quadrilateral outside of $\mathbf{Q}$).
If instead $S$ is a great circle, in order to preserve $C$ as a set it
must (a) coincide with $C$, (b) coincide with $C^\perp$, or (c) intersect $C$ (and so also $C^\perp$) orthogonally.
Clearly neither $C$ nor $C^\perp$ is contained in $M$, since, for example, neither
$\overline{\pqq_0 \pqq_1}$ nor $\overline{\pqq^0 \pqq^1}$ is contained in $\partial D_0^0$.
In case (c), in order to permute $\mathbf{Q}$, $S$ can be only one of the circles 
listed in (ii) (a circle containing an edge of a quadrilateral in $\mathbf{Q}$)
or one of the circles of intersection of a pair of spheres of symmetry
(a circle bisecting the edges on $C$ and $C^\perp$ of some $Q_i^j$, not necessarily having $i+j$ even),
but none of these latter circles is contained in $M$,
since, for example, for all $i,j \in \Z$ $\pqq_{i+\frac12} \not \in \partial D_i^j$.
\end{proof} 

\begin{cor}[Umbilics on the Lawson surfaces]
\label{umb}
For $M=M[C,m]$ as in \ref{T:lawson} we have only four umbilics, $\pqq^1$, $\pqq^2$, $\pqq^3$, and $\pqq^4$,   
of degree (as in \cite{Lawson}) $m-2$ each. 
\end{cor} 

\begin{proof} 
By the symmetries it is clear that each of these point is an umbilic of degree $m-2$. 
By a result of Lawson \cite{Lawson}*{Proposition 1.5} 
the total degree of the umbilics is $4g-4=4m-8$ and so there can be no other.
\end{proof}

\begin{cor}[The unit normal on the geodesic segments $\overline{ \pqq_i\pqq^1 } $] 
\label{Lnui}
By appropriate choice of the unit normal 
$\nu:M\to\Sph^3$ smoothly defined on $M$ we have $\forall i\in\Z$ 
$$
\nu\left( \, \overline{\pqq_{i}\pqq^1} \, \right) =
\overline{ \, \pqq_{i - (-1)^i\frac{m}2} \, \pqq^0 \, } = 
\overline{ \, \pp_{\frac{i}{m}\pi     -(-1)^i\frac\pi2 - \frac{\pi}{2m}            } \, \pp^{-\pi/4} \, }. 
$$
\end{cor}

\begin{proof}
A unit vector normal to a great circle $C' \subset \Sph^3$
lies on the circle $C'^\perp$.
By the symmetries (Lemma \ref{Msym})
the unit normal $\nu$ on $M \cap C$ must point along $C$,
while on $M \cap C^\perp$ it must point along $C^\perp$.
Thus $\nu(\pqq_1)=\pm \pqq_{1+m/2}$ and $\nu(\pqq^1)=\pm \pqq^0$.
Assume that
$\nu(\pqq_1)=\pqq_{1+m/2}$.
Using Lemma \ref{Msym} again,
it suffices to complete the proof for $i=1$.
Since $M$ is disjoint from the interior of $\Omega_1^0$ (and $\Omega_0^1$),
we conclude that along all of $\overline{\pqq_1\pqq^1}$
the normal $\nu$ cannot cross either $\Sigma_{\pi/2m}$ or $\Sigma^{\pi/4}$
and more specifically, by our choice of $\nu(\pqq_1)$, must point into $\Omega_1^0$.
It follows that $\nu(\pqq^1)=\pqq^0$
and $\nu(x) \cdot \nu(\pqq_1) \geq 0$ and $\nu(x) \cdot \nu(\pqq^1) \geq 0$
for all $x \in \overline{\pqq_1\pqq^1}$,
completing the proof.
\end{proof}

Although it is not needed in this article,
we include the following lemma to offer a fuller picture of the symmetry group.

\begin{lemma}[Further symmetries of the Lawson surfaces]
\label{Msym2}
For $M=M[C,m]$ as in \ref{T:lawson} we have the following symmetries.  
\\
(i) A great circle $C' \not \subset M$ is a circle of symmetry for $M$
    if and only if (a) $C'=C$, (b) $C'$ is one of the $2m$ circles
    $C_{i\pi/m}^{j\pi/2}$ having $i,j \in \Z$,
    or (c) $m$ is even and $C'=C^\perp$.
    Each such $\refl_{C'}$ preserves $\nu$.
\\
(ii) The antipodal map $\refl_{\emptyset}$ belongs to $\Gsym^M$ if and only if $m$ is even,
     in which case $\refl_{\emptyset}$ preserves $\nu$.
\\
(iii) A point $x \in \Sph^3$ is a point of symmetry for $M$
     ($\refl_x \in \Gsym^M$) if and only if
     (a) $x$ is one of the $2m$ points $\pp_{i\pi/m}$ with $i \in \Z$
     or (b) $x$ is one of the $4$ points $\pp^{j\pi/4}$ with $j \in 2\Z+m$.
     Each such $\refl_x$ reverses $\nu$.
\\
(iv) For every $i \in \Z$ the map
$\refl_{\pqq_i} \circ \rot_C^{\pi/2}  
\in \Gsym^M$ and reverses $\nu$. 
\end{lemma} 

\begin{proof}
It is easy to see that for any $i \in \Z$
both $\refl_{\pqq_i}$ and $\rot_C^{\pi/2}$
exchange the sets $\left\{\Omega_i^j\right\}_{i+j \in 2\Z}$
and $\left\{\Omega_i^j\right\}_{i+j \in 2\Z+1}$,
so that the composite acts as a permutation on each of these sets
and therefore (as explained in the proof of Lemma \ref{Msym})
belongs to $\Gsym^M$;
it is also easy to see that the composite reverses the normal at $\pqq_i$,
completing the proof of (iv).
Item (ii) follows from (i),
since $\refl_{\emptyset}=\refl_C\refl_{C^\perp}$.
The fact that the circles listed in (i) 
exhaust all circles of symmetry not lying on $M$
follows from the final paragraph of the proof of Lemma \ref{Msym}.
The rest of (i) is easily proven using Lemma \ref{Msym} itself
(and the group structure of $O(4)$) as follows.
Clearly $\refl_{\Sigma'} \circ \refl_{\Sigma''}=\refl_{\Sigma' \cap \Sigma''}$
for any two great spheres $\Sigma'$ and $\Sigma''$ intersecting orthogonally.
On the other hand, according to Lemma \ref{Msym},
the great $2$-spheres of symmetry of $M$ are precisely the spheres $\Sigma_{i\pi/m}$ and $\Sigma^{j\pi/2}$
for $i,j \in \Z$, so in particular $\Sigma_{\pi/2}$ is a sphere of symmetry precisely when $m$ is even.
Together, the preceding two sentences complete the proof of (i).

To prove (iii)
first note that for any point $x \in \Sph^3$
the set $x^\perp$
is the round $2$-sphere centered at $\pm x$,
and
moreover
$\refl_x=\refl_{-x}=-\refl_{x^\perp}=\refl_{\Sigma_x} \circ \refl_{C_x}$,
where $\Sigma_x$ is any great sphere through $\pm x$ and $C_x$ is any
great circle orthogonally intersecting $\Sigma_x$ at $\pm x$.
In particular $\refl_x \in \Gsym^M$ precisely when $\refl_{x^\perp}$
takes $M$ to $-M$.
Since $-D_i^j=D_{i+m}^{j+2}$,
we have $-M=\bigcup_{i+j \equiv m \bmod 2} D_i^j$.
It is clear from Lemma \ref{Msym} that $\Grp_{sym}^M$ preserves $C$ as a set 
(since each generator obviously does so). Thus in order for $x$ to be a point
of symmetry of $M$ (whatever the parity of $m$) $x$ must lie on either $C$ or $C^\perp$.
Since $C$ is itself a great circle of symmetry,
any point of symmetry lying on $C$ must also lie on a sphere of symmetry intersecting $C$ orthogonally.
Thus the set of points of symmetry lying on $C$
is simply $\left\{\pp_{i\pi/m}\right\}_{i \in \Z}$.
To identify the points of symmetry on $C^\perp$
we observe by Lemma \ref{Msym}
that $\Grp_{sym}^M$ preserves the set $\{\pqq^j\}_{j \in \Z}$,
which means that a point of symmetry on $C^\perp$ must lie in
$\{\pqq^j\}_{j \in \frac12 \Z}=\{\pp^{j\pi/4}\}_{j \in \Z}$.
It is easy to see that $\refl_{\Sigma^{j \pi/4+\pi/2}}$
takes $M$ to $-M$ precisely when $j-m \in 2\Z$,
which completes the proof.
\end{proof}

\subsection*{Graphical properties} 
$\phantom{ab}$
\nopagebreak

\begin{lemma}[Graphical property and subdivisions of $D_i^j$] 
\label{Dsym}
$\forall i,j\in\Z$ the following hold. 
\\ 
(i)  
$D_i^j$ is graphical---with its interior strongly graphical---with respect to 
$K^{C_{i\pi/m}^{j\pi/2} } = K_{C_{i\pi/m+\pi/2}^{j\pi/2+\pi/2} }$ (recall \ref{graphical})  
and each orbit which intersects $\Om_i^j$ intersects $D_i^j$ as well. 
\\ 
(ii)  
Each of $D_{i\pm}^{j} := D_i^j \cap  \Om_{i\pm}^{j}$, $D_{i}^{j\pm} := D_i^j \cap  \Om_{i}^{j\pm}$, and $D_{i\pm}^{j\pm} := D_i^j \cap  \Om_{i\pm}^{j\pm}$, 
is homeomorphic to a closed disc.  
\end{lemma} 

\begin{proof} 
To prove (i) we first prove that $D_i^j$ is graphical.  
This follows by the same argument 
as in the second paragraph of the proof of \ref{T:lawson} but with ${D'}_i^j$ replaced by $D_i^j$.
Consider now the Jacobi field $\nu\cdot K_{C_{i\pi/m+\pi/2}^{j\pi/2+\pi/2} }$,  
which clearly by the graphical property and appropriate choice of $\nu$ is $\ge0$ on $D_i^j$ and hence by the maximum principle is $>0$ on the interior of $D_i^j$. 
This implies that the interior of $D_i^j$ is strongly graphical. 

Next we 
recall the projection map 
\begin{equation} 
\label{Piij} 
\Pi_i^j:= \Pi^{C_{i\pi/m}^{j\pi/2}}_{\pp_{i\pi/m} }  
: \Om_i^j \to C_{i\pi/m+\pi/2}^{j\pi/2+\pi/2} \cone \pp_{i\pi/m} 
\end{equation} 
defined in \ref{L-alex}.vi.  
Let $D':=\Pi_i^j(\Om_i^j)$, 
which by \ref{L-alex}.vi is homeomorphic to a closed disc with $\partial D'= \Pi_i^j(Q_i^j)$.  
Clearly then $\Pi_i^j( D_i^j )\subset D'$.  
Since $\partial D_i^j=Q_i^j$ we have also 
$\Pi_i^j( \partial D_i^j ) = \partial D'$,  
and therefore 
$\Pi_i^j( D_i^j ) = D'$, 
which completes the proof of (i).  

Furthermore, as shown above, $D_i^j$ is graphical with respect to $K^{C_{i\pi/m}^{j\pi/2}}$,
so the restriction $\Pi_i^j|_{D_i^j}$ is one-to-one.
We conclude that $\Pi_i^j$ takes $D_i^j$ homeomorphically onto 
$D'$. 
The proof of (ii) is then completed by the fact that $\Pi_i^j$ clearly respects the symmetries of $\Omega_i^j$.
\end{proof}

By the definitions when $i+j\in 2\Z$ we have  
$M \cap  \Om_{i\pm}^{j\pm} = D_{i\pm}^{j\pm}$; 
otherwise we have 
$M \cap  \Om_{i\pm}^{j\pm} = \emptyset$. 
By \ref{Dsym} 
each $D_{i\pm}^{j\pm}$ is an embedded minimal disc. 
To study $\partial D_{i\pm}^{j\pm}$ and the intersections with two-spheres of symmetry we define 
the intersections of $D_i^j$ and $D_{i\pm}^{j\pm}$ with the bisecting two-spheres as follows. 
\begin{equation} 
\label{ab} 
\begin{aligned} 
\alpha_i^{j\pm} :=& \, 
D_i^j \cap \overline{ \, \pqq_{ i+\frac12 } \pqq^{j+\frac12} \pqq^{j+\frac12\pm\frac12} \, } = 
D_i^{j\pm} \cap \Sigma_{ i \frac\pi{m}  } ,   
\\
\alpha_i^{j} :=& \, 
D_i^j \cap \overline{ \, \pqq_{ i+\frac12 } \pqq^{ j } \pqq^{ j+1 } \, } = 
D_i^{j} \cap \Sigma_{ i \frac\pi{m}  } = 
\alpha_i^{j - }  \cup \alpha_i^{j + } ,   
\\ 
\beta_{i\pm}^j :=& \, 
D_i^j \cap \overline{ \, \pqq_{i+\frac12} \pqq_{i+\frac12\pm\frac12} \pqq^{ j+\frac12 } \, } =  
D_{i\pm}^j \cap \Sigma^{ j \frac\pi{2}  }  ,  
\\ 
\beta_{i}^j :=& \, 
D_i^j \cap \overline{ \,  \pqq_{ i }  \pqq_{ i+1 } \pqq^{ j+\frac12 } \, } =  
D_{i}^j \cap \Sigma^{ j \frac\pi{2}  }  = 
\beta_{i - }^j \cup \beta_{i + }^j  .  
\end{aligned} 
\end{equation}

\begin{lemma}[The $\alpha$ and $\beta$ curves] 
\label{Dpm} 
$\forall i,j\in\Z$ the following hold. 
\\
(i)  
$D_i^j$ intersects $\overline{ \pp_{ i\frac\pi m }  \pp^{ j\frac\pi 2 } } = 
\overline{ \pqq_{ i+\frac12 }  \pqq^{ j+\frac12  } } 
$ 
at a single point which we will call $\QQ_i^j$. 
\\ 
(ii)  
The sets $\alpha_i^{j - }$, $\alpha_i^{j + }$, $\beta_{i - }^j$, $\beta_{i + }^j$, $\alpha_i^{j}$, and $\beta_{i}^j$  
are connected curves with 
$\partial \alpha_i^{j - } = \{  \pqq^{ j } , \QQ_i^j \}$, 
$\partial \alpha_i^{j + } =  \{  \pqq^{ j+ 1} , \QQ_i^j \}$, 
$\partial \beta_{i - }^j =  \{ \pqq_{ i } , \QQ_i^j \}$, 
$\partial \beta_{i + }^j =  \{ \pqq_{ i+ 1 } , \QQ_i^j \}$, 
$\partial \alpha_i^{j} =  \{ \pqq^j, \pqq^{j+1 } \}$, 
and 
$\partial \beta_{i}^j = \{ \pqq_i, \pqq_{i+1 } \}$. 
\\ 
(iii)  
$\partial D_{i\pm}^{j\pm} 
\, = \, \overline{ \pp_{ (2i \pm 1)\frac\pi{2m} } \, \pp^{ (2j \pm 1)\frac\pi4  } }  
\cup 
\alpha_i^{j\pm} 
\cup 
\beta_{i\pm}^j
\, = \, \overline{ \pqq_{ i + \frac12 \pm \frac12 } \, \pqq^{ j + \frac12 \pm \frac12 } }  
\cup 
\alpha_i^{j\pm} 
\cup 
\beta_{i\pm}^j.$ 
\end{lemma} 
  
\begin{proof}
As in the previous proof we consider $\Pi_i^j$, 
which is a homeomorphism from $D_i^j$ onto $D'$ and moreover respects the symmetries of $\Om_i^j$. 
Using the various definitions it is then straightforward to complete the proof. 
\end{proof}

\begin{lemma}[Graphical properties of $D_{i\pm}^{j\pm}$] 
\label{Dsym2}
$\forall i,j\in\Z$ the following hold. 
\\
(i) 
The interior of $D_i^j$ is contained in the interior of $\Om_i^j$ and the conormal of $D_i^j$ at a point of $Q_i^j\setminus \Qv_i^j$ 
is transverse to each face of $\Om_i^j$ containing the point. 
\\
(ii)  
$D_{i\pm}^{j}$ (as in \ref{Dsym}(ii)) is graphical with respect to $K_{C^\perp}$ and strongly graphical in its interior.  
\\ 
(iii)  
$D_{i}^{j\pm}$ (as in \ref{Dsym}(ii)) is graphical with respect to $K_{C}$ and strongly graphical in its interior.  
\end{lemma} 

\begin{proof} 
(i) follows easily by the maximum principle 
\cite{schoen1983}*{Lemma 1}.    
The proofs of (ii) and (iii) are based on Alexandrov reflection in the style of \cite{schoen1983}. 
Clearly $ \Pi^C_{\pqq_i } \Om_i^j = \overline{ \, \pqq_{ i } \pqq^{ j } \pqq^{ j+1 } \, } $ 
by \ref{Pi} and \ref{Om:p}. 
For $t\in[0,\pi/m]$ we define (recall \ref{qpoints})  
$$ 
\pqq_{i,t} := \pp_{t_i+t},  
\qquad 
D_{i,t}^j :=D_i^j \cap \overline{ \, \pqq_{ i } \pqq_{i,t} \pqq^{ j } \pqq^{ j+1 } \, } ,  
\qquad \text{ and } \qquad 
D_{i:t}^j :=D_i^j \cap \overline{ \, \pqq_{ i+1 } \pqq_{i,t} \pqq^{ j } \pqq^{ j+1 } \, }.   
$$ 
We clearly have then $D_i^j = D_{i,t}^j \cup D_{i:t}^j$ and $D_{i,t}^j \cap D_{i:t}^j = D_i^j \cap \overline{ \, \pqq_{ i,t } \pqq^{ j } \pqq^{ j+1 } \, } $.  

Clearly 
$\Om_i^j\setminus D_i^j$ consists of two connected components, 
which in this proof we call $U_1$ and $U_2$, 
where 
$U_1$ is chosen to be the component which contains the interior of 
$ \overline{ \, \pqq_{ i } \pqq^{ j } \pqq^{ j+1 } \, } $.  
We define 
$$
T:= \{ t\in (0,\pi/m) : U_1 \cap \refl_{ C^\perp , \pqq_{i,t} }  D_{i,t}^j \ne  \emptyset \} . 
$$
(i) implies that $D_{i,t}^j$ is graphical for $t$ small enough, 
and therefore  
$t\notin T$ for $t$ small enough.    
We conclude that $t':=\inf T>0$.  
If $t'<\frac\pi{2m}$, then by the definition of $t'$, 
$\refl_{ C^\perp , \pqq_{i,t'} }  D_{i,t'}^j$ 
and 
$D_{i:t'}^j$ have a point of one-sided interior or boundary tangential contact. 
By the maximum principle 
\cite{schoen1983}*{Lemma 1} and analytic continuation this implies that 
$\refl_{ C^\perp , \pqq_{i,t'} }  D_{i,t'}^j$ 
and 
$D_{i:t'}^j$ are identical contradicting the symmetries of $Q_i^j$ (alternatively \ref{Msym}). 
We conclude that $t'\ge\frac\pi{2m}$ and hence 
$$
U_1 \cap \refl_{ C^\perp , \pqq_{i,t} }  D_{i,t}^j = \emptyset 
\qquad \text{ for } 
t\in (0, \pi/2m ). 
$$ 

Using this we prove now that 
$D_{i-}^{j}$ 
is graphical with respect to $K_{C^\perp}$:  
Otherwise there would be an orbit which would contain two points $\pyy_1\ne\pyy_2$ with 
$\pyy_i\in D_i^j \cap \overline{ \, \pqq_{ i,t_i } \pqq^{ j } \pqq^{ j+1 } \, } $  
for $i=1,2$, where $0< t_1<t_2\le \frac\pi{2m} $. 
$\pyy_2$ is then a point of interior one-sided tangential contact of 
$\refl_{ C^\perp , \pqq_{i,t_*} }  D_{i,t_*}^j$  
and $D_i^j$, where $t_*=\frac{t_1+t_2}2\in (0,\frac\pi{2m} )$. 
This implies that 
$\refl_{ C^\perp , \pqq_{i,t_*} }  $  
is a symmetry of $D_i^j$, and hence of $\partial D_i^j = Q_i^j$, which is a contradiction. 
To prove that it is strongly graphical in the interior we argue as in the proof of \ref{Dsym}. 
By symmetry we conclude the statement for 
$D_{i+}^{j}$. 
This completes the proof of (ii). 
The proof of (iii) is similar with the roles of $C$ and $C^\perp$ exchanged.  
\end{proof}

We define $[m:2]:= 0$ if $m\in2\Z$ and $[m:2]:= 1$ otherwise. 

\begin{lemma}[Some intersections of $M$ with great two-spheres]  
\label{SigmaM} 
We have the following.  
\newline
(i)  
$ 
M\cap \Sph (C^\perp,\pqq_i) 
= 
M\cap \Sigma_{(2i- 1)\frac\pi{2m}  \,} 
= 
\bigcup_{j\in\Z} 
\Sph(\pqq_i,\pqq^j) 
= 
\bigcup_{j\in\Z} 
C_{ (2i- 1)\frac\pi{2m} }^{ (2j- 1)\frac\pi4  \,}$ 
$\forall i \in \Z$. 
\\ 
(ii)  
$ 
M\cap \Sph (C^\perp,\pqq_{i+\frac12} ) 
= 
M\cap \Sigma_{i\frac\pi{m}  \,} = \bigcup_{j\in2\Z-i}  \alpha_i^j \cup \alpha_{i+m}^{j+[m:2]} $  
$\forall i \in \Z$.  
\newline
(iii)  
$ 
M\cap \Sph (C,\pqq^j) 
= 
M\cap \Sigma^{(2j- 1)\frac\pi{4}  \,} 
= 
\bigcup_{i\in\Z} 
\Sph(\pqq_i,\pqq^j) 
= 
\bigcup_{i\in\Z} 
C_{ (2i- 1)\frac\pi{2m} }^{ (2j- 1)\frac\pi4  \phantom{{\frac12}^A}}$  
$\forall j \in \Z$. 
\\ 
(iv)  
$ 
M\cap \Sph (C,\pqq^{j+\frac12}) 
= 
M\cap \Sigma^{j\frac\pi{2}  \,} = \bigcup_{i\in2\Z-j}  \beta_i^j \cup  \beta_{i}^{j+2}$ 
$\forall j \in \Z$. 
\end{lemma} 

\begin{proof} 
That the circles are contained in the intersections in (i) and (iii) follows from the definition of $M$ in \ref{T:lawson}
and the reverse inclusions follow from \ref{Dsym2}.i  
completing the proof of (i) and (iii). 
By \ref{L:obs}.iv we have 
$ \Sigma_{i\frac\pi{m} \,} =  ( C^\perp \cone \pp_{i\frac\pi{m}\,}  ) \cup ( C^\perp \cone \pp_{i\frac\pi{m}+\pi\,}  )$  
and 
$\Sigma^{j\frac\pi{2}  \,} = ( C \cone \pp^{j\frac\pi{2}  \,} ) \cup ( C \cone \pp^{j\frac\pi{2}+\pi  \,} )$.  
By \ref{ab} and \ref{T:lawson} we have  
$M\cap ( C^\perp \cone \pp_{i\frac\pi{m}\,}  ) 
=
\bigcup_{j\in2\Z-i}  
\alpha_i^j $ 
and 
$M\cap ( C \cone \pp^{j\frac\pi{2}  \,} ) 
=
\bigcup_{i\in2\Z-j}  
\beta_i^j $ . 
Using \ref{qpoints} we complete the proof of (ii) and (iv). 
\end{proof}

\subsection*{Subdividing the Lawson surfaces with mutually orthogonal two-spheres} 
$\phantom{ab}$
\nopagebreak

\begin{definition}
\label{Mpm} 
For $M=M[C,m]$ as in \ref{T:lawson} 
we define $M^{\pm\pm}_{\pm\pm}:= M \cap \Omu^{\pm\pm}_{\pm\pm}$, 
where instead of $\pm$ we could also have $*$ 
(recall \ref{Omu+-}).  
For example 
$M_{+-}^{-*} := M \cap \Omu_{+-}^{-*}$.  
\qed 
\end{definition} 

\begin{lemma}[Description of $M^{++}_{**}$]  
\label{M++}
The following hold. 
\\ 
(i)  
$M^{\pm\pm}_{**}$ is homeomorphic to a closed disc and   
$M^{++}_{**} = \cup_{i=0}^{m-1} ( D^{0+}_{2i} \cup D^{1-}_{2i+1} )$.  
\\
(ii)  
$\partial M^{++}_{**} = ( \Sigma^0\cap M^{++}_{**} ) \cup ( \Sigma^{\pi/2} \cap M^{++}_{**} ) $ 
and is homeomorphic to a circle. 
\\
(iii)  
$\Sigma^0\cap M^{++}_{**} = \cup_{i=0}^{m-1}  \beta^{0}_{2i} $,  
and so consists of $m$ connected components, 
each homeomorphic to a closed interval.  
\\ 
(iv)  
$\Sigma^{\pi/2} \cap M^{++}_{**} = \cup_{i=0}^{m-1} \beta^{1}_{2i+1} $, 
and so consists of $m$ connected components, 
each homeomorphic to a closed interval.  
\\ 
(v)  
$\Sigma_0\cap M^{++}_{**}$ is homeomorphic to a closed interval and 
$\Sigma_0\cap M^{++}_{**} = \left\{ 
\begin{aligned}
& \alpha_0^{0+} \cup \alpha_m^{0+} \text{ if } m\in 2\Z,  
\\
& \alpha_0^{0+} \cup \alpha_m^{1-} \text{ if } m\in 2\Z+1.  
\end{aligned} 
\right. 
\qquad 
$ 
\\ 
(vi)  
$\Sigma_{\pi/2} \cap M^{++}_{**}$ is homeomorphic to a closed interval and 
\\ $\phantom{kk}$ \hfill 
$\Sigma_{\pi/2} \cap M^{++}_{**} = \left\{ 
\begin{aligned}
& \alpha_{m/2}^{0+}  \cup \alpha_{3m/2}^{0+}  && \text{ if } m\in 4\Z,  
\\
& \overline{\pqq^1\pqq_{\frac{m+1}2}} \cup  \overline{\pqq^1\pqq_{\frac{3m+1}2}} && \text{ if } m\in 2\Z+1,  
\\
& \alpha_{m/2}^{1-}  \cup \alpha_{3m/2}^{1-}  && \text{ if } m\in 4\Z+2.  
\end{aligned} 
\right. 
\qquad 
$ 
\end{lemma} 

\begin{proof}
All items follow easily from \ref{OmuOm}, \ref{T:lawson}, \ref{ab}, \ref{Dpm}, and \ref{Mpm}. 
\end{proof} 

\begin{lemma}[Description of $M^{++}_{++}$]  
\label{M++++}
The following hold. 
\\ 
(i)  
$M^{++}_{++} = D_{0+}^{0+} \cup 
\left\{ 
\begin{aligned}
& \cup_{i=1}^{\frac{m}{2}-1} D^{[i:2]\pm}_i \cup D_{\frac{m}2-}^{0+} && \qquad \text{ if } m\in 4\Z, \qquad   
\\
& \cup_{i=1}^{\frac{m-1}{2}} D^{[i:2]\pm}_i && \qquad \text{ if } m\in 2\Z+1, \qquad   
\\
& \cup_{i=1}^{\frac{m}{2}-1} D^{[i:2]\pm}_i \cup D_{\frac{m}{2}-}^{1-} && \qquad \text{ if } m\in 4\Z+2, \qquad   
\end{aligned} 
\right. 
$
\\
where the $\pm$ signs are $+$ for $i$ even and $-$ for $i$ odd. 
Therefore $M^{++}_{++}$ is homeomorphic to a closed disc.  
\\
(ii)  
$\partial M^{++}_{++}$ is homeomorphic to a circle. 
Moreover we can write  
$\partial M^{++}_{++} = 
\gamma_1\cup\gamma_2\cup\gamma_3$, where   
$\gamma_1:= \gamma_{1-} \cup \gamma_{1+}$,  
$\gamma_{1-}:= \Sigma^{0} \cap M^{++}_{++} $, 
$\gamma_{1+}:= \Sigma^{\pi/2} \cap M^{++}_{++} $, 
$\gamma_2:= \Sigma_0\cap M^{++}_{++}$, and   
$\gamma_3:= \Sigma_{\pi/2} \cap M^{++}_{++} $,  
and each of $\gamma_1, \gamma_2,\gamma_3$ is homeomorphic to a closed interval.  
\\
(iii)  
$\gamma_1 = \gamma_{1-} \cup \gamma_{1+} = 
\beta_{0+}^{0} \cup 
\left\{ 
\begin{aligned}
& \cup_{i=1}^{\frac{m}{2}-1} \beta^{[i:2]}_i \cup \beta_{\frac{m}2-}^{0} && \text{ if } m\in 4\Z,  
\\
& \cup_{i=1}^{\frac{m-1}{2}} \beta^{[i:2]}_i && \text{ if } m\in 2\Z+1,  
\\
& \cup_{i=1}^{\frac{m}{2}-1} \beta^{[i:2]}_i \cup \beta_{\frac{m}{2}-}^1 && \text{ if } m\in 4\Z+2,  
\end{aligned} 
\right. 
$
\\
(iv)  
$\gamma_{1-}  = \Sigma^0 \cap M^{++}_{++} = 
\overline{\pp_{0} \pp_{\pi/2} \pp^0}  \cap M^{++}_{++} = 
\beta_{0+}^{0} \cup 
\left\{ 
\begin{aligned}
& \cup_{i=1}^{\frac{m-4}{4}} \beta_{2i}^0 \cup \beta_{\frac{m}2-}^{0} && \text{ if } m\in 4\Z,  
\\
& \cup_{i=1}^{\frac{m-1}{4}} \beta_{2i}^0 && \text{ if } m\in 4\Z+1,  
\\
& \cup_{i=1}^{\frac{m-2}{4}} \beta_{2i}^0 && \text{ if } m\in 4\Z+2,  
\\ 
& \cup_{i=1}^{\frac{m-3}{4}} \beta_{2i}^0 && \text{ if } m\in 4\Z+3.  
\end{aligned} 
\right. 
$
\\ 
(v)  
$\gamma_{1+} = \Sigma^{\pi/2} \cap M^{++}_{++} = 
\overline{\pp_{\pi/2} \pp^{\pi/2} \pp_0} \cap M^{++}_{++} = 
\left\{ 
\begin{aligned}
& \cup_{i=1}^{m/4} \beta_{2i-1}^1 && \text{ if } m\in 4\Z,  
\\
& \cup_{i=1}^{\frac{m-1}{4}} \beta_{2i-1}^1 && \text{ if } m\in 4\Z+1,  
\\
& \cup_{i=1}^{\frac{m-2}{4}} \beta_{2i-1}^1 \cup \beta_{\frac{m}{2}-}^1 && \text{ if } m\in 4\Z+2,  
\\ 
& \cup_{i=1}^{\frac{m+1}{4}} \beta_{2i-1}^1 && \text{ if } m\in 4\Z+3.  
\end{aligned} 
\right. 
$
\\ 
(vi)  
$\gamma_2 = \Sigma_0 \cap M^{++}_{++} = 
\overline{\pp_{0} \pp^{\pi/2} \pp^0} \cap M^{++}_{++} = 
\alpha_0^{{{0+}}} $. 
\\ 
(vii)  
$\gamma_3 = \Sigma_{\pi/2} \cap M^{++}_{++} = 
\overline{\pp_{\pi/2} \pp^{\pi/2} \pp^0} \cap M^{++}_{++} = 
\left\{ 
\begin{aligned}
& \alpha_{m/2}^{0+}  && \text{ if } m\in 4\Z,  
\\
& \overline{\pqq^1\pqq_{\frac{m+1}2}}      =       \overline{\pp^{\pi/4} \pp_{\pi/2} }   && \text{ if } m\in 2\Z+1,  
\\
& \alpha_{m/2}^{1-}  && \text{ if } m\in 4\Z+2,  
\end{aligned} 
\right. 
$ 
\end{lemma}

\begin{proof}
All items follow easily from \ref{OmuOm}, \ref{T:lawson}, \ref{ab}, \ref{Dpm}, and \ref{Mpm}. 
\end{proof}

\begin{lemma}[Description of $M^{++}_{+*}$]  
\label{M+++}
The following hold. 
\\ 
(i)  
$M^{++}_{+*} = D_{0+}^{0+} \cup 
\left\{ 
\begin{aligned}
& \cup_{i=1}^{m-1} D^{[i:2]\pm}_i \cup D_{m-}^{0+} && \qquad \text{ if } m\in 2\Z, \qquad   
\\
& \cup_{i=1}^{m-1} D^{[i:2]\pm}_i \cup D_{m-}^{1-} && \qquad \text{ if } m\in 2\Z+1, \qquad   
\end{aligned} 
\right. 
$
\\
where the $\pm$ signs are $+$ for $i$ even and $-$ for $i$ odd. 
Therefore $M^{++}_{+*}$ is homeomorphic to a closed disc.  
\\
(ii)  
$\partial M^{++}_{+*}$ is homeomorphic to a circle. 
Moreover we can write  
$\partial M^{++}_{+*} = 
\gamma_4\cup\gamma_5$, where   
$\gamma_4:= \gamma_{4-} \cup \gamma_{4+}$,  
$\gamma_{4-}:= \Sigma^{0} \cap M^{++}_{+*} $, 
$\gamma_{4+}:= \Sigma^{\pi/2} \cap M^{++}_{+*} $, 
$\gamma_5:= \Sigma_0\cap M^{++}_{+*}$, 
and each of $\gamma_4, \gamma_5$ is homeomorphic to a closed interval.  
\\
(iii)  
$\gamma_4 = \gamma_{4-} \cup \gamma_{4+} = 
\beta_{0+}^{0} \cup 
\left\{ 
\begin{aligned}
& \cup_{i=1}^{m-1} \beta^{[i:2]}_i \cup \beta_{m-}^{0} && \text{ if } m\in 2\Z,  
\\
& \cup_{i=1}^{m-1} \beta^{[i:2]}_i \cup \beta_{m-}^1 && \text{ if } m\in 2\Z+1,  
\end{aligned} 
\right. 
$
\\
(iv)  
$\gamma_{4-}  = \Sigma^0 \cap M^{++}_{+*} = 
\beta_{0+}^{0} \cup 
\left\{ 
\begin{aligned}
& \cup_{i=1}^{\frac{m-2}{2}} \beta_{2i}^0 \cup \beta_{{m}-}^{0} && \text{ if } m\in 2\Z,  
\\
& \cup_{i=1}^{\frac{m-1}{2}} \beta_{2i}^0 && \text{ if } m\in 2\Z+1.  
\end{aligned} 
\right. 
$
\\ 
(v)  
$\gamma_{4+} = \Sigma^{\pi/2} \cap M^{++}_{+*} = 
\left\{ 
\begin{aligned}
& \cup_{i=1}^{m/2} \beta_{2i-1}^1 && \text{ if } m\in 2\Z,  
\\
& \cup_{i=1}^{\frac{m-1}{2}} \beta_{2i-1}^1 \cup \beta_{m-}^1 && \text{ if } m\in 2\Z+1.  
\end{aligned} 
\right. 
$
\\ 
(vi)  
$\gamma_5 = \Sigma_0 \cap M^{++}_{+*} = 
\Sigma_0\cap M^{++}_{**} = \left\{ 
\begin{aligned}
& \alpha_0^{0+} \cup \alpha_m^{0+} \text{ if } m\in 2\Z,  
\\
& \alpha_0^{0+} \cup \alpha_m^{1-} \text{ if } m\in 2\Z+1.  
\end{aligned} 
\right. 
\qquad 
$ 
\\ 
(vii)  
$\Sigma_{\pi/2} \cap M^{++}_{+*} = 
\Sigma_{\pi/2} \cap M^{++}_{++} = 
\left\{ 
\begin{aligned}
& \alpha_{m/2}^{0+}  && \text{ if } m\in 4\Z,  
\\
& \overline{\pqq^1\pqq_{\frac{m+1}2}}      =       \overline{\pp^{\pi/4} \pp_{\pi/2} }   && \text{ if } m\in 2\Z+1,  
\\
& \alpha_{m/2}^{1-}  && \text{ if } m\in 4\Z+2,  
\end{aligned} 
\right. 
$ 
\end{lemma} 

\begin{proof}
All items follow easily from \ref{OmuOm}, \ref{T:lawson}, \ref{ab}, \ref{Dpm}, and \ref{Mpm}. 
\end{proof}

\begin{lemma}[Symmetries of $M^{++}_{++}$ and $M^{++}_{+*}$]  
\label{symM++++}
$\pp_{\frac\pi2} = \pqq_{\frac{m+1}2 }$ and for $m$ even $\pp_{\frac\pi4} = \pqq_{\frac{m}4 + \frac12}$. 
Moreover the following hold. 
\\ 
(i)  
If $m\in 4\Z$ then $\refl_{\Sigma_{\pi/4}} = \refl_{\pqq_{\frac{m}4 + \frac12} , C^\perp }$ is a symmetry of  $M^{++}_{++}$ preserving the unit normal.   
\\ 
(ii)  
If $m\in 4\Z+2$ then $\refl_{ C_{\pi/4}^{\pi/4} } = \refl_{\pqq_{ \frac{m+2}4 } , \pqq^1} $ is a symmetry of  $M^{++}_{++}$ reversing the unit normal.   
\\ 
(iii)  
If $m\in 2\Z+1$ then $\refl_{ C_{\pi/2}^{\pi/4} } = \refl_{ \pqq_{ \frac{m+1}2 } , \pqq^1 } $ is a symmetry of $M^{++}_{+*}$ 
reversing the unit normal and exchanging $M^{++}_{++}$ with $M^{++}_{+-}$.   
\end{lemma}

\begin{proof}
All items follow easily from \ref{Omu:p}.iii,  \ref{Msym}, \ref{Mpm}, and  \ref{M++++}. 
\end{proof}

\section{Jacobi fields}

As is well known the linearized operator for the mean curvature on 
$M=M[C,m]$ (recall \ref{T:lawson}) 
is given by 
\begin{equation} 
\label{Lcal} 
\Lcal := \Delta + |A|^2 + 2 ,
\end{equation} 
where $|A|^2$ is the square of the length of the second fundamental form of the surface. 
We recall next the following standard definition. 

\begin{definition}[Jacobi fields on {{$M=M[C,m]$}}] 
\label{jacobiM} 
We call a function $J\in C^\infty(M)$ a \emph{Jacobi field on $M=M[C,m]$} if it satisfies $\Lcal J = 0$. 
\qed 
\end{definition} 

It is well known that Killing fields induce Jacobi fields as in the following definition. 

\begin{definition}[Jacobi fields $J_{C'}$] 
\label{jacobi} 
We adopt now for the rest of this article the same choice for the unit normal $\nu$ 
on the Lawson surface $M=M[C,m]$ as in \ref{Lnui}.  
Given then a great circle $C'$ in $\Sph^3$ and assuming an orientation on ${C'}^\perp$,  
we define the \emph{Jacobi field} $J^{{C'}^\perp}= J^{{C'}^\perp}[C,m] = J_{C'}= J_{C'}[C,m] \in C^\infty (\, M[C,m]\,)$ by 
$
J^{{C'}^\perp}= J_{C'} := K_{C'} \cdot \nu  
$ 
(recall \ref{l:D:rot}).  
\qed 
\end{definition} 

Note that multiplying a Jacobi field by $-1$ 
changes neither its nodal set nor any other significant properties, 
and so the orientation of ${C'}^\perp$ and direction of $\nu$ 
do not play a fundamental role. 

\begin{definition}[Exceptional and non-exceptional Jacobi fields] 
\label{exceptional} 
We call a Jacobi field on $M$ \emph{non-exceptional} if it is induced by a Killing field;  
otherwise we call it \emph{exceptional}. 
\qed 
\end{definition} 

We proceed to study some non-exceptional Jacobi fields which as we will see in \ref{basis-killing} form a basis. 
It is useful to introduce first the following notation. 

\begin{notation} 
\label{N:perp} 
We define $0_\perp:= \pi/2$ and $(\pi/2)_\perp:=0$. 
\qed 
\end{notation} 

\begin{lemma}[Symmetries of Jacobi fields]  
\label{Jsym} 
$\forall i,j\in\Z$ we have the following. 
\\
(i)  
$J_C$ is odd under $\refl_{\Sigma^{j\pi/2}} = \refl_{\, C , \pqq^{j+\frac12} } $ 
and even under $\refl_{\Sigma_{i\pi/m}} = \refl_{\, C^\perp, \pqq_{i+\frac12} } $ 
and 
$\refl_{ C_{ (2i- 1)\frac\pi{2m} }^{ (2j- 1)\frac\pi4  \,} } 
= \refl_{\, \pqq^{j} , \pqq_{i} } $.  
\\
(ii)  
$J_{C^\perp}$ is odd under $\refl_{\Sigma_{i\pi/m}} = \refl_{\, C^\perp, \pqq_{i+\frac12} } $ 
and even under $\refl_{\Sigma^{j\pi/2}} = \refl_{\, C , \pqq^{j+\frac12} } $ 
and 
$\refl_{ C_{ (2i- 1)\frac\pi{2m} }^{ (2j- 1)\frac\pi4  \,} } 
= \refl_{\, \pqq^{j} , \pqq_{i} } $.  
\\
(iii)  
If $m\in2\Z$ and $\phi,\phi'\in\{0,\pi/2\}$, 
then $J_{C_\phi^{\phi'} }$ is odd under $\refl_{\Sigma_\phi}$ and $\refl_{\Sigma^{\phi'}}$ 
and even under $\refl_{\Sigma_{\phi_\perp}}$ and $\refl_{\Sigma^{\phi'_\perp}}$.  
\\
(iv)  
If $m\in2\Z+1$, then the symmetries in (iii) hold except for the ones associated with $\refl_{\Sigma_{\pi/2}}$. 
\end{lemma} 

\begin{proof} 
All items follow from \ref{Ksymm} and \ref{Msym}. 
Note that $\refl_{\Sigma_{\pi/2}}$ preserves $M$ only when $m$ is even. 
\end{proof}

\begin{lemma}[Action of some symmetries on some Jacobi fields] 
\label{AsymJ}
The following hold. 
\\ 
(i)  
If $m\in 4\Z$, then 
$J_{C^0_0} \circ \refl_{\Sigma_{\pi/4}} = - J_{C^0_{\pi/2}}$ and 
$J_{C^{\pi/2}_{\pi/2}} \circ \refl_{\Sigma_{\pi/4}} = J_{C_0^{\pi/2}}$. 
\\ 
(ii)  
If $m\in 4\Z+2$, then 
$J_{C^0_0} \circ \refl_{ C_{\pi/4}^{\pi/4} } = J_{C^{\pi/2}_{\pi/2}}$ and  
$J_{C^0_{\pi/2}} \circ \refl_{ C_{\pi/4}^{\pi/4} } = - J_{C_0^{\pi/2}}$. 
\\ 
(iii)  
If $m\in 2\Z+1$, then 
$J_{C^0_0} \circ \refl_{ C_{\pi/2}^{\pi/4} } =  J_{C_0^{\pi/2}}$       
and 
$J_{C^{\pi/2}_{\pi/2}} \circ \refl_{ C_{\pi/2}^{\pi/4} } =  J_{C^0_{\pi/2}} $. 
\end{lemma}

\begin{proof}
All items follow easily from \ref{symM++++} and the definitions, using in particular the orientation conventions specified in \ref{E:killing}. 
\end{proof}

\begin{lemma}[Gradient of Jacobi fields at $\pqq^1$]  
\label{L:grad} 
If $\phi,\phi'\in\{0,\pi/2\}$, 
then 
$J_{C_\phi^{\phi'} }(\pqq^1)=0$ 
and the gradient 
$\nabla_{\pqq^1}J_{C_\phi^{\phi'} }$ at $\pqq^1$ 
is nonzero and tangential to $\Sigma_{\phi_\perp}$. 
\end{lemma} 

\begin{proof} 
By \ref{umb} $M$ has high-order contact with $\Sigma^{\pi/4}$ at $\pqq^1$,  
so we can consider the corresponding Jacobi field on $\Sigma^{\pi/4}$ instead. 
That Jacobi field is clearly a first harmonic of $\Sigma^{\pi/4}$ 
and the result follows without calculation by the symmetries. 
\end{proof}

\begin{lemma}[Non-exceptional Jacobi fields] 
\label{basis-killing} 
$J_{C}$, 
$J_{C^\perp}$, 
and 
$J_{C_\phi^{\phi'} }$  
for 
$\phi,\phi'\in\{0,\pi/2\}$  
form a basis of the space of non-exceptional Jacobi fields on $M=M[C,m]$.  
\end{lemma} 

\begin{proof} 
Since the space of Killing fields has dimension six, 
it is enough to prove that the Jacobi fields under consideration are linearly independent. 
By symmetrizing and antisymmetrizing with respect to 
$\refl_{\Sigma^0}$, $\refl_{\Sigma^{\pi/2}}$, and (for the last four) $\refl_{\Sigma_0}$, 
we can kill all of them by \ref{Jsym} except for a single Jacobi field arbitrarily chosen in advance. 
This reduces the proof to proving that each of the six Jacobi fields does not vanish identically. 
Clearly $J_C(\pqq^1)\ne0$, $J_{C^\perp}(\pqq_1)\ne0$, and for the rest we consider the gradient 
at $\pqq^1$ and we appeal to \ref{L:grad}. This completes the proof. 

An alternative proof is that 
the map $K \mapsto K \cdot \nu$ from the space of Killing fields
to the space of Jacobi fields is injective: 
if not, 
there would exist a nontrivial Killing field everywhere tangential to $M$,
meaning $M$ would have a one-parameter family of symmetries.
By Lemma \ref{Msym}, however, $\Gsym^M$ is discrete, completing the proof. 
\end{proof}

\begin{lemma}[Some Jacobi fields on geodesic segments] 
\label{rays} 
We have the following. 
\\
(i)  
For $i\in (2\Z+1) \cap [1,(m+1)/2]$ 
and 
$i\in (2\Z) \cap [(m+1)/2, m ]$  
we have 
$J_{C_0^0} \ge0$ on 
$\overline{ \pqq_i \pqq^1 } \subset M^{++}_{+*}$.  
\\
(ii)  
For $i\in (2\Z) \cap [1,(m+1)/2]$ 
and 
$i\in (2\Z+1) \cap [(m+1)/2, m ]$  
we have 
$J_{C_0^{\pi/2}} \ge0$ on 
$\overline{ \pqq_i \pqq^1 } \subset M^{++}_{+*}$.  
\\
(iii)  
For $i\in (2\Z) \cap [1,(m+1)/2]$ 
we have 
$J_{C^0_{\pi/2}} \le0$ on 
$\overline{ \pqq_i \pqq^1 } \subset M^{++}_{++}$  
and 
\\ 
for $i\in (2\Z+1) \cap [(m+1)/2, m ]$  
we have 
$J_{C^0_{\pi/2}} \ge0$ on 
$\overline{ \pqq_i \pqq^1 } \subset M^{++}_{+-}$.  
\\
(iv)  
For $i\in (2\Z+1) \cap [1,(m+1)/2]$ 
we have 
$J_{C^{\pi/2}_{\pi/2}} \ge0$ on 
$\overline{ \pqq_i \pqq^1 } \subset M^{++}_{++}$  
and 
\\ 
for $i\in (2\Z) \cap [(m+1)/2, m ]$  
we have 
$J_{C^{\pi/2}_{\pi/2}} \le0$ on 
$\overline{ \pqq_i \pqq^1 } \subset M^{++}_{+-}$.  
\end{lemma} 

\begin{proof} 
By \ref{Omu:K} and \ref{Mpm}  
  \begin{equation}
    \begin{aligned}
      &\left.K_{C_0^0}\right|_{M_{+\pm}^{++}}
        \mbox{ is a convex combination of }
        \pp_{\pi/2} \mbox{ and } -\pp^{\pi/2}, \\
      &\left.K_{C_0^{\pi/2}}\right|_{M_{+\pm}^{++}}
        \mbox{ is a convex combination of }
        -\pp_{\pi/2} \mbox{ and } \pp^0, \\
      &\left.K_{C_{\pi/2}^0}\right|_{M_{+\pm}^{++}}
        \mbox{ is a convex combination of }
        -\pp_0 \mbox{ and } \pm \pp^{\pi/2}, \mbox{ and} \\
      &\left.K_{C_{\pi/2}^{\pi/2}}\right|_{M_{+\pm}^{++}}
        \mbox{ is a convex combination of }
        -\pp_0 \mbox{ and } \pm \pp^0.
    \end{aligned}
  \end{equation}

Meanwhile, according to Lemma \ref{Lnui},
at each point on $\overline{\pqq_i\pqq^1}$ the unit normal $\nu$
is a convex combination of
$\nu(\pqq_i)=\pp_{\frac{2i-1}{2m}\pi+(-1)^{i+1}\frac{\pi}{2}}$
and $\nu(\pqq^1)=\pp^{-\pi/4}$.
Note that
  \begin{equation}
      \nu(\pqq_i) \in
      \begin{cases}
        \overline{\pp_{\pi/2}\pp_\pi} &\mbox{for } i \in (2\Z+1) \cap [1,(m+1)/2] \\
        \overline{\pp_{-\pi/2}\pp_0} &\mbox{for } i \in (2\Z) \cap [1,(m+1)/2] \\
        \overline{\pp_\pi \pp_{-\pi/2}} &\mbox{for } i \in (2\Z+1) \cap [(m+1)/2,m] \\
        \overline{\pp_0 \pp_{\pi/2}} &\mbox{for } i \in (2\Z) \cap [(m+1)/2,m],
      \end{cases}
  \end{equation}
so on $M^{++}_{+*}$ 
  \begin{equation}
    \pp_0 \cdot \nu(\pqq_i) \in
      \begin{cases}
        [0,1] &\mbox{for } i \in 2\Z \\
        [-1,0] &\mbox{for } i \in 2\Z+1
      \end{cases}
  \end{equation}
and
  \begin{equation}
    \pp_{\pi/2} \cdot \nu(\pqq_i) \in
      \begin{cases}
        [0,1] &\mbox{for } i \in \left((2\Z+1) \cap [1,(m+1)/2]\right) \cup \left((2\Z) \cap [(m+1)/2,m]\right) \\
        [-1,0] &\mbox{for } i \in \left((2\Z) \cap [1,(m+1)/2]\right) \cup \left((2\Z+1) \cap [(m+1)/2,m]\right).
      \end{cases}
  \end{equation}

On the other hand, obviously
  \begin{equation}
    \pp^0 \cdot \nu(\pqq^1) \in [0,1]
    \quad \mbox{and} \quad
    \pp^{\pi/2} \cdot \nu(\pqq^1) \in [-1,0],
  \end{equation}
while of course
  \begin{equation}
    \pp^0 \cdot \nu(\pqq_i)=\pp^{\pi/2} \cdot \nu(\pqq_i)=\pp_0 \cdot \nu(\pqq^1)=\pp_{\pi/2} \cdot \nu(\pqq^1)=0.
  \end{equation}
All items now follow from the convexity and the signs of the inner products recorded above.
\end{proof}

We define now a kind of discrete derivative $\Dcal$ for functions on $M$ by appropriately adapting to the current 
situation the discrete derivative defined in \cite{alm20}*{(8.13), page 319}: 

\begin{definition}[$\Tr$ and the discrete derivative $\Dcal$] 
\label{D:TD} 
We define an isometry $\Tr:\Sph^3\to\Sph^3$ by $\Tr:= \refl_{\pqq_1,\pqq^1} \circ \refl_{\Sigma_0}$ and 
a linear map $\Dcal:C^\infty(M)\to C^\infty(M)$ by $\Dcal f := \frac1{2 \sin(\pi/m)} ( f\circ \Tr - f \circ \Tr^{-1} ) $  
$\forall f\in C^\infty(M)$. 
\qed 
\end{definition} 

\begin{lemma}[Elementary properties of $\Tr$ and $\Dcal$] 
\label{L:T} 
$\Dcal$ as in \ref{D:TD} is well defined and  
$\Tr$ preserves $C$, $C^\perp$, and $M=M[C,m]$,  
and on $M$ satisfies $\Tr^{-1}_*\circ \nu \circ \Tr = -\nu$. 
Moreover $\Tr =\rot^C_{\pi/m} \circ \refl_{\Sigma^{\pi/4}}$ and so $\Tr$ rotates $C$ along itself by angle $\pi/m$ and reflects $C^\perp$ to itself 
while fixing $\pqq^1=\pp^{\pi/4}$ and $\pqq^3=-\pp^{\pi/4}$.  
\end{lemma} 

\begin{proof} 
The first statement about $\Tr$ follows from \ref{Msym}.  
It follows then that $\Dcal$ is well defined.  
Using the definitions it is easy to check that $\Tr$ maps $\pp_0, \pp_{\pi/2}, \pp^0, \pp^{\pi/2}$ to 
$\pp_{\pi/m} , \pp_{\pi/2 + \pi/m}, \pp^{\pi/2}, \pp^0$  
respectively. 
This implies the last statement and completes the proof. 
\end{proof} 

\begin{lemma}[Discrete derivatives of some Jacobi fields] 
\label{L:D} 
The following hold. 
\\ 
(i)  $\Dcal J_{C_0^0} = \, J_{C_{\pi/2}^{\pi/2} }$ and $\Dcal J_{C_{\pi/2}^{\pi/2} } = \, - J_{C_0^0} $. 
\\ 
(ii)  $\Dcal J_{C_0^{\pi/2} } = \, - J_{C_{\pi/2}^{0} }$ and $\Dcal J_{C_{\pi/2}^0 } = \, J_{C_{0}^{\pi/2} }$. 
\end{lemma} 

\begin{proof} 
Note that 
if $J=K\cdot\nu$ is a Jacobi field induced by a Killing field $K$, 
then 
$J\circ\Tr= ( K\circ\Tr ) \cdot ( \nu\circ\Tr ) = ( \Tr^{-1}_*\circ K \circ\Tr ) \cdot ( \Tr^{-1}_*\circ \nu \circ \Tr ) 
= - ( \Tr^{-1}_*\circ K \circ\Tr ) \cdot \nu$ 
and similarly  
$J\circ\Tr^{-1} = - ( \Tr_*\circ K \circ\Tr^{-1} ) \cdot \nu$, 
so we have   
\begin{equation} 
\label{E:Tr} 
\Dcal J = \textstyle{ \frac1{2 \sin(\pi/m)} } \, ( - \Tr^{-1}_*\circ K \circ\Tr + \Tr_*\circ K \circ\Tr^{-1} ) \cdot \nu. 
\end{equation} 
By \ref{L:T} we have 
\begin{equation} 
\begin{aligned} 
\Tr ( \, y^1 \pp_0 + y^2 \pp_{\pi/2} + y^3 \pp^0 + y^4 \pp^{\pi/2} \, ) \, =& \, (cy^1-sy^2) \pp_0 + (sy^1+cy^2)\pp_{\pi/2} + y^4\pp^0 + y^3 \pp^{\pi/2}, 
\\ 
\Tr^{-1} ( \, y^1 \pp_0 + y^2 \pp_{\pi/2} + y^3 \pp^0 + y^4 \pp^{\pi/2} \, ) \, =& \, (cy^1+sy^2) \pp_0 + (-sy^1+cy^2)\pp_{\pi/2} + y^4\pp^0 + y^3 \pp^{\pi/2}, 
\end{aligned} 
\end{equation} 
where in this proof we simplify the notation by taking $c:=\cos\frac\pi{m}$ and $s:=\sin\frac\pi{m}$.  
It is easy to calculate then by referring to \ref{E:killing} that 
\begin{equation*} 
\begin{aligned} 
\Tr_*^{-1} \circ K_{C_0^0} \circ \Tr \, ( \, x^1 \pp_0 + x^2 \pp_{\pi/2} + x^3 \pp^0 + x^4 \pp^{\pi/2} \, ) \, =& 
\, \phantom{-} s x^3 \pp_0 + cx^3 \pp_{\pi/2} - (sx^1+cx^2) \pp^0,  
\\ 
\Tr_*^{-1} \circ K_{C_{\pi/2}^{\pi/2} } \circ \Tr \, ( \, x^1 \pp_0 + x^2 \pp_{\pi/2} + x^3 \pp^0 + x^4 \pp^{\pi/2} \, ) \, =& 
\, -c x^4 \pp_0 + sx^4 \pp_{\pi/2} + (cx^1-sx^2) \pp^{\pi/2},  
\\ 
\Tr_*^{-1} \circ K_{C_{0}^{\pi/2} } \circ \Tr \, ( \, x^1 \pp_0 + x^2 \pp_{\pi/2} + x^3 \pp^0 + x^4 \pp^{\pi/2} \, ) \, =& 
\, -s x^4 \pp_0 - cx^4 \pp_{\pi/2} + (sx^1+cx^2) \pp^{\pi/2},  
\\ 
\Tr_*^{-1} \circ K_{C^{0}_{\pi/2} } \circ \Tr \, ( \, x^1 \pp_0 + x^2 \pp_{\pi/2} + x^3 \pp^0 + x^4 \pp^{\pi/2} \, ) \, =& 
\, -c x^3 \pp_0 + sx^3 \pp_{\pi/2} + (cx^1-sx^2) \pp^{0},  
\end{aligned} 
\end{equation*} 
If we exchange $\Tr$ and $\Tr^{-1}$ in the left hand sides we obtain the same expressions but with ``$s$'' replaced by ``$-s$''. 
Subtracting, applying \ref{E:Tr}, and referring to \ref{E:killing} again, we conclude the proof. 
\end{proof}

\section{Eigenfunctions on the Lawson surfaces} 

In this section we study the index and nullity of the linear operator $\Lcal$ on $M=M[C,m]$ defined in \ref{Lcal}. 
$\Lcal$ is the only operator we consider in this section and so we often omit it in order to simplify the notation, 
especially in the notation of \ref{D:mixed}. 
We start by defining 
\begin{equation}
\label{Vpm} 
V^{\pm\pm}:= \{u\in C_{pw}^\infty(M)  \, : \, u\circ \refl_{\Sigma^0}=\pm u \text{ and } u\circ\refl_{\Sigma^{\pi/2}}=\pm u\,\}, 
\end{equation} 
where the $\pm$ signs are taken correspondingly, the first one referring to $\refl_{\Sigma^0}$ and the second one to $\refl_{\Sigma^{\pi/2}}$.
We clearly have 
\begin{equation} 
\label{V++}
V = V^{++} \oplusu V^{+-} \oplusu V^{-+} \oplusu V^{--}, 
\end{equation}  
where we use $\oplusu$ to mean ``direct sum'' not only in the sense of linear spaces, 
but also to mean that the summands are invariant under $\Lcal$, 
and therefore the same decomposition holds for the corresponding eigenspaces. 

\begin{prop} 
\label{P:V--} 
We have the following (recall \ref{D:eigen} and \ref{D:eigen-eq}).
\\ (i) 
$J_C\in V^{--}$ and  
$V^{--}\sim C_{pw}^\infty[M^{++}_{**}; \partial M^{++}_{**} ,\emptyset ]$.  
\\ (ii) 
$\ind(V^{--}) = 0$ 
and   
$\nul(V^{--}) = 1$.  
\end{prop} 

\begin{proof} 
(i) follows from the definitions, 
\ref{M++}, and \ref{Jsym}, where the linear isomorphism is given by restriction to $M^{++}_{**}$ in one direction 
and its inverse by extension by appropriate reflections. 
For (ii) recall first 
that Lemma \ref{Jsym} implies that 
$J_{C}$ is nonnegative on $M^{++}_{**}$ by \ref{Dsym2}.iii and the symmetries. 
$J_{C}$ is nontrivial by \ref{basis-killing} and therefore, 
as a consequence of Courant's nodal theorem \ref{Ctheorem},  
there are no other eigenfuctions in $V^{--}$ of the same or lower eigenvalue as the eigenvalue of $J_{C}$, 
which is zero. 
The result follows. 
\end{proof} 

To study 
$V^{++}$ now we define 
\begin{equation}
\label{V++pm} 
\begin{aligned} 
V^{++}_{\pm} :=& \{u\in V^{++} \, : \, \forall i\in\Z \,\, \,\, \, \, u\circ \refl_{\Sigma_{i\pi/m}} = \pm u \,\}, 
\\
V^{++}_{\pm\pm} :=& \{u\in V^{++}_{\pm} \, : \, \forall i,j\in\Z \,\, \,\, \, u\circ \refl_{\, \pqq^{j} , \pqq_{i} } = \pm u\,\}, 
\end{aligned} 
\end{equation} 
where in the second equation the $\pm$ signs are taken correspondingly. 
Note that 
\begin{equation} 
\label{V+++o}
V^{++}_{+} = V^{++}_{++} \oplusu V^{++}_{+-}  
\qquad \text{and} \qquad 
V^{++}_{-} = V^{++}_{-+} \oplusu V^{++}_{--} .  
\end{equation}  
On the other hand $V^{++}$ is not the direct sum of 
$V^{++}_{+}$ 
and 
$V^{++}_{-}$.     

\begin{lemma} 
\label{V++-} 
The following hold (recall \ref{Dsym}.ii and \ref{ab}). 
\\ (i) 
$J_{C^\perp}\in V_{-+}^{++}$ and 
$ 
V^{++}_{-} \sim 
C_{pw}^\infty[ D_{0+}^{0+}\cup D_{1-}^{1-} ; \alpha_1^{1-} \cup \alpha_0^{0+} , \beta_{0+}^0\cup \beta_{1-}^1 ] 
$. 
\\ (ii) 
$\lambda_1(V^{++}_{-})=0<\lambda_2(V^{++}_{-})$. 
\end{lemma} 

\begin{proof} 
$J_{C^\perp}\in V_{-+}^{++}$ follows from \ref{Jsym}.ii and the definitions. 
Recall now 
that 
$D_{0+}^{0+}\cup D_{1-}^{1-}$ is homeomorphic to a closed disc and its boundary is $\beta_{0+}^0\cup \beta_{1-}^1\cup \alpha_1^{1-} \cup \alpha_0^{0+}$. 
We clearly have (i) then, 
where the linear isomorphism is given by restriction in one direction and its inverse by extending using reflections. 
On $D_{0+}^{0+}$ $J_{C^\perp}$ is nonnegative by \ref{Dsym2}.ii and nontrivial by \ref{basis-killing}.  
By \ref{Jsym}.ii it is then nonnegative on $D_{1-}^{1-}$ as well. 
As a consequence of 
Courant's nodal theorem \ref{Ctheorem} 
$J_{C^\perp}$ corresponds then to the lowest eigenvalue and the proof is complete. 
\end{proof}

\begin{lemma} 
\label{V+++} 
The following hold. 
\\ (i) 
$V^{++}_{+-} \sim C_{pw}^\infty [ D_{0+}^{0+}; \overline{\pqq^1\pqq_1} , \beta_{0+}^0\cup \alpha_0^{0+} ] \sim  
\{\, u\in C_{pw}^\infty[ D_0^0 ; \partial D_0^0 , \emptyset ] \, : \, 
u \circ \refl_{\Sigma_0} = u \circ \refl_{\Sigma^0} = u \, \}$.  
\\ (ii) 
$\lambda_1(V^{++}_{+-})>0$ and 
$\lambda_1(V^{++}_{+})<0<\lambda_2(V^{++}_{+})$. 
\end{lemma} 

\begin{proof} 
(i) follows easily by the symmetries with linear isomorphisms being restrictions in one direction and inverses given by extensions using 
even or odd reflections appropriately. 
By \ref{V+++o} to prove (ii) it is enough to prove 
\begin{equation} 
\label{l3} 
\lambda_1(V^{++}_{+-})>0, \qquad 
\lambda_1(V^{++}_{++})<0, \qquad 
\lambda_2(V^{++}_{++})>0. 
\end{equation} 

Let $\phi_1$ be an eigenfunction corresponding to the eigenvalue $\lambda_1(V^{++}_{+-})$. 
Since $D_0^0$ is a minimizing disc, it is weakly stable, and so $\lambda_1(V^{++}_{+-})\ge0$. 
The first inequality in \ref{l3} will follow if we prove $\lambda_1(V^{++}_{+-})\ne0$. 
By Courant's nodal theorem \ref{Ctheorem} $\phi_1$ cannot change sign on $D_0^0$ and so by the maximum principle 
(without loss of generality) $\phi_1>0$ on the interior of $D_0^0$ and $\eta \phi_1<0$ on 
$\partial D_0^0\setminus \Qv_0^0= Q_0^0\setminus \Qv_0^0$, 
where $\eta$ is the outward unit conormal derivative of $D_0^0$ at $\partial D_0^0\setminus \Qv_0^0$.  
By applying Green's second identity to $\phi_1$ and $J_{C_{\pi/2}^{\pi/2}}$ we conclude 
$$
\int_{D_0^0} ( J_{C_{\pi/2}^{\pi/2}} \Lcal\phi_1 - \phi_1 \Lcal J_{C_{\pi/2}^{\pi/2}} ) = 
\int_{\partial D_0^0} ( J_{C_{\pi/2}^{\pi/2}} \eta \phi_1 - \phi_1 \eta J_{C_{\pi/2}^{\pi/2}} ) . 
$$ 
If $\lambda_1(V^{++}_{+-})=0$, then the left hand side vanishes. 
Since $\phi_1$ satisfies the Dirichlet condition on $\partial D_0^0$,  
we conclude 
$$ 
\int_{\partial D_0^0}  J_{C_{\pi/2}^{\pi/2}} \,\, \eta \phi_1 \, = \, 0. 
$$ 
By \ref{Dsym}.i $J_{C_{\pi/2}^{\pi/2}}$ does not change sign on $D_0^0$. 
Since $\eta \phi_1<0$ on $\partial D_0^0\setminus \Qv_0^0$, we conclude that $J_{C_{\pi/2}^{\pi/2}}=0$ on $\partial D_0^0$. 
This contradicts \ref{L:grad} (alternatively it implies odd symmetries which together with \ref{Jsym} 
contradict the nontriviality of $J_{C_{\pi/2}^{\pi/2}}$) and therefore we conclude the first inequality in \ref{l3}.  
The positivity of the zeroth-order term of $\Lcal$ 
implies the second inequality in \ref{l3}. 

Suppose now that $\lambda_2(V^{++}_{++})\le0$ and let $\phi_2$ be a corresponding eigenfunction. 
Then $\phi_2$ satisfies Neumann conditions on $\partial D_{0+}^{0+}$
and moreover by Courant's nodal theorem \ref{Ctheorem} will have two nodal domains on $D_{0+}^{0+}$.
It follows from \cite{Cheng}*{Theorem 2.5} that the nodal set $\phi_2^{-1}(\{0\})$
contains (at least) a piecewise $C^2$ embedded circle or segment
whose endpoints (if it has any) lie on $\partial D_{0+}^{0+}$
but which is otherwise disjoint from $\partial D_{0+}^{0+}$.
In particular this nodal curve separates $D_{0+}^{0+}$
into two components and misses the interior of at least one of the three sides---$\beta_{0+}^0$,
$\alpha_0^{0+}$, and $\overline{\pqq^1\pqq_1}$---of $\partial D_{0+}^{0+}$.
We call the missed side $\gamma$.
By domain monotonicity we conclude that 
$\lambda_1[D_{0+}^{0+}; \gamma, \partial D_{0+}^{0+} \setminus \gamma \, ] <0 $.  
If $\gamma= \overline{\pqq^1\pqq_1}$, this would contradict the first inequality in \ref{l3}, which already has been proved. 
If $\gamma= \alpha_0^{0+}$, this would contradict \ref{V++-}. 
Finally if $\gamma= \beta_{0+}^0$, this would contradict \ref{P:V--}, and the proof is complete. 
\end{proof}

\begin{prop} 
\label{P:V++} 
We have $\ind(V^{++}) = 2m- 1$ 
and   
$\nul(V^{++}) = 1$.  
\end{prop} 

\begin{proof} 
By considering $M^{++}_{**}$ and subdividing along the curves of intersection with $\Sigma_{i\pi/m}$ 
and imposing the Dirichlet or the Neumann conditions appropriately, 
the result follows from \ref{P:eigen} by using \ref{V++-} and \ref{V+++}.  
\end{proof} 

To study 
$V^{+-}$ and $V^{-+}$ now we define 
\begin{equation}
\label{V+-pm} 
V^{+-}_{\pm} := \{u\in V^{+-} \, : \, u\circ \refl_{\Sigma_{0}} = \pm u \,\}, 
\qquad 
V^{-+}_{\pm} := \{u\in V^{-+} \, : \, u\circ \refl_{\Sigma_{0}} = \pm u \,\}. 
\end{equation} 
Note that 
\begin{equation} 
\label{V+-+o}
V^{+-} = V^{+-}_{+} \oplusu V^{+-}_{-},  
\qquad  
V^{-+} = V^{-+}_{+} \oplusu V^{-+}_{-}.  
\end{equation}

\begin{lemma}[The sign of some Jacobi fields]   
\label{MsJ} 
The following hold. 
\\ (i) 
If $m\ge3$, then 
$J_{C^0_0}\ge0$ and $J_{C^{\pi/2}_0}\ge0$ on $M^{++}_{+*}$.  
\\ (ii) 
If in addition $m$ is even, then  
$J_{C^{\pi/2}_{\pi/2}}\ge0$ and $J_{C_{\pi/2}^0}\le0$ 
on $M^{++}_{++}$.  
\end{lemma} 
 
\begin{proof} 
\emph{Step 1: 
We prove that 
$\forall i\in \Z\cap [1,(m+1)/2]$  
we have $J_{C^0_0}\ge0$ on $\overline{ \pqq_i \pqq^1 } $---equivalently 
$J_{C^0_0}\ge0$ on all geodesic segments contained in $M^{++}_{++}$. } 

If $i$ is odd or $i=\frac{m+1}2$, 
we already know this by \ref{rays}.i. 
We can assume then that $m\ge4$ because for $m=3$ step 1 is proved. 
By \ref{rays}.iv and \ref{L:D}.i we have for 
$i\in (2\Z+1) \cap [1,(m+1)/2]$ that 
$J_{C^0_0} \circ \Tr \ge  J_{C^0_0} \circ \Tr^{-1} $ on 
$\overline{ \pqq_{i} \pqq^1 } $. 
Since $\Tr^{\pm 1}(\overline{\pqq_i\pqq^1})=\overline{\pqq_{i\pm 1}\pqq^1}$ by \ref{L:T}, 
this means that $J_{C^0_0}$ on $\overline{ \pqq_{i} \pqq^1 } \subset M^{++}_{++}$ is increasing with increasing even $i$. 
Arguing inductively on even $i$ it is enough to prove then that 
$J_{C^0_0}\ge0$ on 
$\overline{ \pqq_{2} \pqq^1 } $. 

Taking $i=1$ in the inequality in the previous paragraph we establish that 
$J_{C^0_0} \circ \Tr \ge  J_{C^0_0} \circ \Tr^{-1} $ on 
$\overline{ \pqq_{1} \pqq^1 } $.  
By \ref{D:TD} we have 
$\Tr^{-1} = \refl_{\Sigma_0}$ on ${\overline{\pqq_1\pqq^1}}$,  
so by \ref{Jsym}.iii-iv we know that  
$J_{C^0_0} \circ \Tr^{-1} = - J_{C^0_0} $ on 
$\overline{ \pqq_{1} \pqq^1 } $.  
Combining we obtain 
$J_{C^0_0} \circ \Tr + J_{C^0_0} \ge 0 $ on $\overline{ \pqq_{1} \pqq^1 } $.  
We consider now the symmetrization $\varphi:= J_{C^0_0} \circ \refl_{\Sigma_{\pi/m} } + J_{C^0_0}$ of $J_{C^0_0}$ on $D^{1-}_1 \subset M^{++}_{++}$.  
Recall that 
$\partial D^{1-}_1 = \overline{ \pqq_{1} \pqq^1 } \cup \overline{ \pqq_{2} \pqq^1 } \cup \beta_1^1$. 
Since by \ref{L:T} $\Tr= \refl_{\Sigma_{\pi/m}}$ on $\overline{ \pqq_{1} \pqq^1 } $,   
by the last inequality above we have $\varphi\ge0$ on 
$\overline{ \pqq_{1} \pqq^1 } \cup \overline{ \pqq_{2} \pqq^1 } \subset \partial D^{1-}_1 $.  
By \ref{Jsym}.iii-iv $\varphi$ satisfies the Neumann condition on the remaining boundary $\beta_1^1$.
If we assume that $\varphi$ attains negative values on $D^{1-}_1$, then by domain monotonicity,  
and since $\Lcal \varphi =0$, 
we obtain a contradiction to $\lambda_1(V_{+-}^{++})>0$ in \ref{V+++}.ii by using \ref{V+++}.i. 
Hence $\varphi\ge0$ on $D^{1-}_1$ and since $\varphi=2 J_{C^0_0}$ on $ \alpha^{1-}_1 = D^{1-}_1 \cap \Sigma_{\pi/m} $ (by \ref{ab}),  
we conclude that $ J_{C^0_0} \ge 0$  on $ \alpha^{1-}_1 \subset D^{1-}_1 $.

We consider now the domain $\Phi := D^{1-}_{1+} \cup D^{0+}_2$. 
Clearly $\Phi$ is homeomorphic to a disc and has $\partial \Phi = \beta^1_{1+} \cup \beta^0_2 \cup  \overline{ \pqq_{3} \pqq^1 } \cup \alpha^{1-}_1$ 
and  $\overline{ \pqq_{2} \pqq^1 } \subset \Phi$. 
We postpone the case $m=4$ for later and we assume that $m\ge5$ so that 
$ \overline{ \pqq_{3} \pqq^1 } \subset \Phi \subset  M^{++}_{++}$.  
We already know then 
(by the preceding paragraphs and \ref{Jsym}) 
that  
$ J_{C^0_0} \ge 0$  on $ \overline{ \pqq_{3} \pqq^1 } \cup \alpha^{1-}_1$ and satisfies the Dirichlet condition on $\beta^0_2$ and the Neumann condition on $\beta^1_{1+}$.   
In order to apply now \ref{P:eigen} 
we subdivide $\Phi$ along 
$\overline{ \pqq_{2} \pqq^1 }$ into  
$D^{1-}_{1+}$ and $D^{0+}_2$. 
We clearly have  
$C_{pw}^\infty[\Lcal, D^{1-}_{1+} ; \alpha^{1-}_1 , \overline{ \pqq_{2} \pqq^1 } \cup  \beta^1_{1+} ] \sim V^{++}_{-+} $  
and 
$\lambda_1[\Lcal,  D^{0+}_2   ; \beta^0_2 \cup  \overline{ \pqq_{3} \pqq^1 } , \overline{ \pqq_{2} \pqq^1 } ]  
> 
\lambda_1[\Lcal,  D^{0+}_2   ; \beta^0_2 , \overline{ \pqq_{2} \pqq^1 } \cup  \overline{ \pqq_{3} \pqq^1 } ] \ge \lambda_1(V^{--}) $.  
By referring to \ref{V++-} and \ref{P:V--} we conclude that 
$\#_{<0} [\Lcal, D^{1-}_{1+} ; \alpha^{1-}_1 , \overline{ \pqq_{2} \pqq^1 } \cup  \beta^1_{1+} ] =0$ 
and 
$\#_{\le0} [\Lcal,  D^{0+}_2   ; \beta^0_2 \cup  \overline{ \pqq_{3} \pqq^1 } , \overline{ \pqq_{2} \pqq^1 } ] =0$. 
By \ref{P:eigen} then we conclude that 
$\lambda_1[\Lcal, \Phi; \beta^0_2 \cup  \overline{ \pqq_{3} \pqq^1 } \cup \alpha^{1-}_1, \beta^1_{1+} ] >0 $, 
which by domain monotonicity would contradict an assumption that $J_{C^0_0}$ takes negative values on $\Phi$.    
We conclude that 
$ J_{C^0_0} \ge 0$  on $\overline{ \pqq_{2} \pqq^1 } \subset \Phi$ 
which completes step 1 under the assumption $m\ge5$. 

We consider now the case $m=4$. 
Note that 
$D^{1-}_{1+} \cup D^{0+}_{2-}$ is homeomorphic to a disc and can be subdivided into 
$D^{1-}_{1+}$ and $D^{0+}_{2-}$ by $\overline{ \pqq_2 \pqq^1 }$. 
Recall that 
$\partial D^{1-}_{1+} = \alpha^{1-}_1 \cup \beta_{1+}^1 \cup \overline{ \pqq_2 \pqq^1 } $ and 
$\partial D^{0+}_{2-} = \alpha^{0+}_2 \cup \beta_{2-}^0 \cup \overline{ \pqq_2 \pqq^1 } $.   
We have $\lambda_1[ D^{1-}_{1+} ; \alpha^{1-}_1 , \beta_{1+}^1 \cup \overline{ \pqq_2 \pqq^1 } ] = 0 $ by \ref{V++-} and 
$\lambda_1[ D^{0+}_{2-} ; \beta_{2-}^0 , \alpha^{0+}_2 \cup \overline{ \pqq_2 \pqq^1 } ] =0$ by \ref{P:V--}. 
Applying \ref{P:eigen} we conclude that 
$\lambda_1[ D^{1-}_{1+} \cup D^{0+}_{2-} ; \alpha^{1-}_1 \cup \beta_{2-}^0 , \beta_{1+}^1 \cup \alpha^{0+}_2 ] \ge 0 $.  
Since for $m=4$ we have $\alpha_2^{0+} \subset \Sigma_{\pi/2}$,
it follows by \ref{Jsym} that 
$ J_{C^0_0} $ satisfies the same boundary conditions except on $\alpha_1^{1-}$, 
where we proved above that it is $\ge0$.  
Moreover 
$ J_{C^0_0} $ cannot vanish identically on $\alpha_1^{1-}$ by \ref{L:grad}. 
By domain monotonicity we obtain a contradiction on the assumption that 
$ J_{C^0_0} $ is not nonnegative on 
$D^{1-}_{1+} \cup D^{0+}_{2-}$. 
We conclude 
$ J_{C^0_0} \ge 0$  on $\overline{ \pqq_{2} \pqq^1 } \subset D^{1-}_{1+} \cup D^{0+}_{2-}$ and step 1 is complete in all cases. 

\emph{Step 2: 
We prove that 
$\forall i\in \Z\cap [1,(m+1)/2]$  
we have $J_{C^{\pi/2}_0}\ge0$ on $\overline{ \pqq_i \pqq^1 } $---equivalently 
$J_{C^{\pi/2}_0}\ge0$ on all geodesic segments contained in $M^{++}_{++}$. } 

Unlike the case of 
$J_{C^0_0}$ we now know this by \ref{rays}.ii when $i$ is even. 
Using the discrete derivative as before, 
we have 
by \ref{rays}.iii and \ref{L:D}.ii 
for 
$i\in (2\Z) \cap [1,(m+1)/2]$ that 
$J_{C^{\pi/2}_0} \circ \Tr \ge J_{C^{\pi/2}_0} \circ \Tr^{-1} $ on 
$\overline{ \pqq_{i} \pqq^1 } $.  
Arguing inductively on odd $i$, 
it is enough to prove then that 
$J_{C^{\pi/2}_0}\ge0$ on 
$\overline{ \pqq_{1} \pqq^1 } $. 
For this we consider the domain $\Phi' := D^{0+}_{0+} \cup D^{1-}_1$. 
Clearly $\Phi'$ is isometric to $\Phi$ in the previous step (in fact $\Phi=\mathsf{T}(\Phi')$) and 
has $\partial \Phi' = \beta^0_{0+} \cup \beta^1_1 \cup  \overline{ \pqq_{2} \pqq^1 } \cup \alpha^{0+}_0$ 
and  $\overline{ \pqq_{1} \pqq^1 } \subset \Phi'$. 
Similarly to the previous step, 
we already know that  
$J_{C^{\pi/2}_0}\ge0$ on $\overline{ \pqq_{2} \pqq^1 }$ and satisfies the Dirichlet condition on $\beta^1_1 \cup \alpha^{0+}_0$  and the Neumann condition on $\beta^0_{0+}$.   
Arguing then as in the previous step, 
we conclude that $J_{C^{\pi/2}_0}\ge0$ on $\Phi'$ and hence on 
$\overline{ \pqq_{1} \pqq^1 } \subset \Phi'$, 
which completes step 2. 

\emph{Step 3: 
We prove that $J_{C^0_0}\ge0$ and $J_{C^{\pi/2}_0}\ge0$ on $M^{++}_{++}$. } 

Recall from \ref{M++++} that $M^{++}_{++}$ can be subdivided along the geodesic segments it contains 
into $D^{[i:2]\pm}_i$'s,  
$D_{0+}^{0+}$, and  $D_{\frac{m}2-}^{0+}$ (for $m\in 4\Z$) or 
$D_{\frac{m}{2}-}^{1-}$ (for $m\in 4\Z+2$).  
We already know that on the geodesic segments in the boundaries of these regions
we have $J_{C_0^0},J_{C_0^{\pi/2}} \geq 0$,
while on the rest of each boundary $J_{C_0^0}$ and $J_{C_0^{\pi/2}}$
satisfy Dirichlet or Neumann conditions.
We also know, using \ref{P:V--} and 
\ref{V+++},
that if we impose the Dirichlet condition on each geodesic segment
and leave the remaining boundary conditions unchanged,
then the corresponding lowest eigenvalue on each region
(obtained by subdividing along the geodesic segments)
is strictly positive.
If we assume then that 
$J_{C^0_0}$ or $J_{C^{\pi/2}_0}$ attains negative values, 
we will have a contradiction by domain monotonicity. 
This completes step 3. 

\emph{Step 4: 
We complete the proof of the lemma.} 

For $m$ even (i) follows from step 3 and the even symmetry with respect to $\refl_{\Sigma_{\pi/2} }$, 
as asserted in \ref{Jsym}.iii,  
which exchanges $M^{++}_{++}$ with $M^{++}_{+-}$.   
For $m$ odd (i) follows from step 3 and 
by using \ref{symM++++}.iii and the identity 
$J_{C^0_0} \circ \refl_{ C_{\pi/2}^{\pi/4} } = J_{C_0^{\pi/2}}$       
from \ref{AsymJ}.iii.  
Finally (ii) follows from step 3 and (for $m\in 4\Z$) 
\ref{symM++++}.i and \ref{AsymJ}.i 
or (for $m\in 4\Z+2$)  
\ref{symM++++}.ii and \ref{AsymJ}.ii. 
\end{proof}

\begin{prop} 
\label{P:V+--} 
We have the following (recall \ref{M+++}). 
\\ (i) 
$J_{C_0^0}\in V_{-}^{-+}$ 
and  
$V_{-}^{-+} \sim C_{pw}^\infty[M^{++}_{+*}; \gamma_{4-} \cup \gamma_5 , \gamma_{4+} ]$.  
\\ (ii) 
$J_{C_{0}^{\pi/2}}\in V_{-}^{+-}$ 
and  
$V_{-}^{+-} \sim C_{pw}^\infty[M^{++}_{+*}; \gamma_{4+} \cup \gamma_5 , \gamma_{4-} ]$.  
\\ (iii) 
$\ind(V_{-}^{-+}  ) = \ind( V_{-}^{+-} )   = 0$ 
and   
$\nul(V_{-}^{-+}  ) =  \nul( V_{-}^{+-} )   = 1$.  
\end{prop} 

\begin{proof} 
(i) and (ii) follow easily from the definitions and the symmetries in \ref{Jsym},  
with the linear isomorphisms between the spaces given by restriction to $M^{++}_{+*}$ 
and their inverses by extending using the appropriate reflections. 
(iii) follows then from \ref{MsJ}.i and \ref{Ctheorem}. 
\end{proof} 

We proceed to study now 
$V_{+}^{-+}$ and $V_{+}^{+-}$.  
One would like to decompose these spaces further, 
but unfortunately it is clear how to do this only when $m$ is even. 
If $m$ is even, we define 
\begin{equation}
\label{Vpm4} 
V^{\circ\circ}_{\circ\pm} := \{u\in V^{\circ\circ}_{\circ} \, : \, u\circ \refl_{\Sigma_{\pi/2} } = \pm u\,\} \quad \text{ for $m$ even}, 
\end{equation} 
where the upper circles can be $+-$ or $-+$ (on both sides) and the lower circle $+$ or $-$ (on both sides). 
We have then for $m$ even that 
\begin{equation} 
\label{V+-+oo}
V^{+-}_{\pm} = V^{+-}_{\pm-} \oplusu V^{+-}_{\pm+},  
\qquad  
V^{-+}_{\pm} = V^{-+}_{\pm-} \oplusu V^{-+}_{\pm+}.  
\end{equation}  
Although we will not use the following lemma, we state it for completeness of exposition---compare also with \ref{AsymJ}.

\begin{lemma}[Some eigenvalue equivalences] 
\label{symV}
The following hold. 
\\ 
(i)  
If $m\in 4\Z$, then 
$V_{-+}^{-+}\sim V^{-+}_{+-}$ 
and 
$V_{-+}^{+-}\sim V^{+-}_{+-}$. 
\\ 
(ii)  
If $m\in 4\Z+2$, then 
$V_{-+}^{-+}\sim V^{+-}_{+-}$, 
and 
$V_{-+}^{+-}\sim V^{-+}_{+-}$. 
\\ 
(iii)  
If $m\in 2\Z+1$, then 
$V^{-+}_{-}\sim V^{+-}_{-}$ 
and 
$V^{-+}_{+}\sim V^{+-}_{+}$. 
\end{lemma} 

\begin{proof} 
All items follow easily from \ref{symM++++} and the definitions. 
\end{proof} 
 
\begin{prop} 
\label{P:V+-+} 
We have the following (recall \ref{M+++} and \ref{M++++}). 
\\ (i) 
$J_{C_{\pi/2}^{0}}\in V_{+}^{-+}$ 
and  
$V_{+}^{-+} \sim C_{pw}^\infty[M^{++}_{+*}; \gamma_{4-} , \gamma_{4+} \cup \gamma_5 ]$.  
Moreover if $m$ is even, we have 
$J_{C_{\pi/2}^{0}}\in V_{+-}^{-+}$ 
and  
$V_{+-}^{-+}\sim C_{pw}^{\infty}[ \, M^{++}_{++} ;  \gamma_{1-}\cup\gamma_3 , \gamma_{1+}\cup \gamma_2 \, ]$.  
\\ (ii) 
$J_{C_{\pi/2}^{\pi/2}}\in V_{+}^{+-}$ 
and  
$V_{+}^{+-} \sim C_{pw}^\infty[M^{++}_{+*}; \gamma_{4+} , \gamma_{4-} \cup \gamma_5 ]$.  
Moreover if $m$ is even, we have 
$J_{C_{\pi/2}^{\pi/2}}\in V_{+-}^{+-}$ 
and  
$V_{+-}^{+-}\sim C_{pw}^{\infty}[ \, M^{++}_{++} ;  \gamma_{1+}\cup\gamma_3 , \gamma_{1-}\cup \gamma_2 \, ]$.  
\\ (iii) 
$\ind(V_{+}^{-+}  ) = \ind( V_{+}^{+-} )   = 1$ 
and   
$\nul(V_{+}^{-+}  ) =  \nul( V_{+}^{+-} )   = 1$.  
\end{prop}

\begin{proof} 
As in the proof of \ref{P:V+--},  
(i) and (ii) follow easily from the definitions and the symmetries in \ref{Jsym},  
with the linear isomorphisms between the spaces given by restriction to $M^{++}_{+*}$ 
and their inverses by extending using the appropriate reflections. 
To prove (iii) now we provide different arguments depending on whether $m$ is even or odd, 
the even case being easier because of the extra symmetry we can employ.

We assume first that $m$ is even. 
By (i), (ii), \ref{MsJ}.ii, and \ref{Ctheorem} we conclude that $\lambda_1( V^{+-}_{+-} ) = \lambda_1( V^{-+}_{+-} ) =0 $,  
$\lambda_2( V^{+-}_{+-} ) >0$, and $\lambda_2( V^{-+}_{+-} ) >0 $.  
Replacing the Dirichlet condition with the Neumann condition reduces the eigenvalues and therefore 
$\lambda_1 ( V^{+-}_{++} ) <  \lambda_1 ( V^{+-}_{+-} ) = 0$   
and 
$\lambda_1 ( V^{-+}_{++} ) <  \lambda_1 ( V^{-+}_{+-} ) = 0$.    
By Courant's nodal theorem \ref{Ctheorem} and arguing as in the proof of \ref{V+++} using 
\cite{Cheng}*{Theorem 2.5}, we conclude that 
the eigenfunction corresponding to 
$\lambda_2 ( V^{+-}_{++} ) $
must contain a separating nodal curve in $M^{++}_{++}$  
which does not intersect at least one of $\gamma_1$, $\gamma_2$, or $\gamma_3$ defined as in \ref{M++++}.ii. 
There is a nodal domain then in $M^{++}_{++}$ which does not intersect 
at least one of $\gamma_1$, $\gamma_2$, or $\gamma_3$.  
If it does not intersect $\gamma_1$, by extending to $M^{++}_{++}$ 
and by using domain monotonicity we conclude that 
$\lambda_2(V^{+-}_{++}) > \lambda_1(V^{--})  $.  
If it does not intersect $\gamma_2$, using domain monotonicity we conclude that 
$\lambda_2(V^{+-}_{++}) > \lambda_1(V^{+-}_{-+})$.  
If it does not intersect $\gamma_3$, using domain monotonicity we conclude that 
$\lambda_2(V^{+-}_{++}) > \lambda_1(V^{+-}_{+-})$.  
Since 
$\lambda_1(V^{--}) =0  $,  
$\lambda_1(V^{+-}_{-+})=0 $,  
and 
$\lambda_1(V^{+-}_{+-})=0 $  
by \ref{P:V--}, \ref{P:V+--}, and the above, 
we conclude that 
$0 < \lambda_2 ( V^{+-}_{++} ) $. 
Arguing similarly we conclude that 
$0 < \lambda_2 ( V^{-+}_{++} ) $.
The above together with the decompositions (by \ref{V+-+oo})  
\begin{equation*} 
V^{+-}_{+} = V^{+-}_{+-} \oplusu V^{+-}_{++},  
\qquad  
V^{-+}_{+} = V^{-+}_{+-} \oplusu V^{-+}_{++}  
\end{equation*}  
imply (iii) in the case that $m$ is even. 

Suppose now that $m$ is odd. 
Recall \ref{M+++}. 
By ``cutting through'' $\overline{ \pqq_1\pqq^1 } $ and $\alpha^{1-}_1$ 
we obtain the decomposition 
$M^{++}_{+*} = D_{0+}^{0+} \cup D_{1-}^{1-} \cup M'$, 
where 
$$ 
M':= D_{1+}^{1-}\cup \cup_{i=2}^{m-1} D^{[i:2]\pm}_i \cup D_{m-}^{1-},    
$$ 
with the signs as in \ref{M+++}.i. 
By \ref{Dsym}.ii $D_{0+}^{0+}$, $D_{1-}^{1-}$, and $M'$ are each homeomorphic to a disc and 
we have 
$\partial D_{0+}^{0+} = \beta^0_{0+} \cup \overline{ \pqq_1\pqq^1 } \cup \alpha^{0+}_0$, 
$\partial D_{1-}^{1-} = \beta^1_{1-} \cup \overline{ \pqq_1\pqq^1 } \cup \alpha^{1-}_1$, 
and 
$\partial M' =  \gammatilde_4\cup \alpha^{1-}_1 \cup \alpha_m^{1-}$, 
where  
$\gammatilde_4:= \gammatilde_{4-} \cup \gammatilde_{4+}$,  
$\gammatilde_{4-} =  \gamma_{4-} \cap M' = \gamma_{4-} \setminus \beta^0_{0+} $,  
and 
$\gammatilde_{4+} = \gamma_{4+} \cap M' = \gamma_{4+} \setminus \beta^1_{1-} $.  
The advantage of $M'$ over $M^{++}_{+*}$ is that $M'$ has an extra symmetry, 
$\refl_{ \pqq_{\frac{m}2+1}, C^\perp} = \refl_{\Sigma_{\frac\pi2 + \frac\pi{2m} } }$, 
which preserves each of $\gammatilde_{4-}$ and $\gammatilde_{4+}$. 
To exploit this we define 
$$ 
W:= C_{pw}^\infty [M'; \gammatilde_{4-}, \gammatilde_{4+} \cup \alpha^{1-}_1 \cup \alpha_m^{1-} ], 
\qquad 
W_\pm:= \{u\in W : u=\pm u \circ \refl_{\Sigma_{\frac\pi2 + \frac\pi{2m} } } \} . 
$$ 
We clearly have then the decomposition $W=W_+\oplusu W_-$.  
We claim now that 
\begin{equation} 
\label{claim} 
\lambda_2(W) = 
\lambda_2[M'; \gammatilde_{4-}, \gammatilde_{4+} \cup \alpha^{1-}_1 \cup \alpha_m^{1-} ] 
>0. 
\end{equation} 

To prove the claim it is enough to prove that $\lambda_1(W_-) >0$ and $\lambda_2(W_+)>0$. 
By ``cutting through'' with 
${\Sigma_{\frac\pi2 + \frac\pi{2m} } }$   
we have the decomposition 
$M'=M'_+\cup\refl_{\Sigma_{\frac\pi2 + \frac\pi{2m} } } M'_+ $, 
where 
$$ 
M'_+:= D^{[\frac{m+1}2:2]\pm}_{\frac{m+1}2+} \cup \cup_{i=\frac{m+3}2}^{m-1} D^{[i:2]\pm}_i \cup D_{m-}^{1-}.   
$$ 
We have then 
$M'_+\cap\refl_{\Sigma_{\frac\pi2 + \frac\pi{2m} } } M'_+ = M'\cap \Sigma_{\frac\pi2 + \frac\pi{2m} } = \alpha^{[\frac{m+1}2:2]\pm}_{\frac{m+1}2}$, 
$\gammatilde_{4\pm} \cap M'_+ = \gamma_{4\pm} \cap M'_+ $,  
$$
\begin{aligned} 
W_-\sim & \, C_{pw}^\infty \left[ M'_+ ; (\gamma_{4-}\cap M'_+ ) \cup \alpha^{[\frac{m+1}2:2]\pm}_{\frac{m+1}2}   , ( \gamma_{4+} \cap M'_+ )  \cup \alpha_m^{1-} \right], 
\\ 
W_+\sim & \, C_{pw}^\infty \left[ M'_+ ; \gamma_{4-}\cap M'_+  ,  \alpha^{[\frac{m+1}2:2]\pm}_{\frac{m+1}2}   \cup ( \gamma_{4+} \cap M'_+ )  \cup \alpha_m^{1-} \right]. 
\end{aligned} 
$$
Next we reposition $M'_+$ by using $\Tr^{-\frac{m+1}2}$ (recall \ref{L:T}) to obtain 
$$ 
M'' := \Tr^{-\frac{m+1}2}  M'_+ = D_{0+}^{0+}\cup \cup_{i=1}^{\frac{m-3}2} D^{[i:2]\pm}_i \cup D_{\frac{m-1}2-}^{[\frac{m-1}2:2]\pm},   
$$   
and we use 
$\refl_{ \pqq_{m/2} , C^\perp \, }  = \refl_{\Sigma_{\frac\pi2 - \frac\pi{2m} } }$ to ``double'' $M''$, 
producing 
$$ 
M''':= 
M'' \cup (\refl_{ \pqq_{m/2} , C^\perp \, } M'' ) 
= D_{0+}^{0+}\cup \cup_{i=1}^{m-2} D^{[i:2]\pm}_i \cup D_{(m-1)-}^{0+}, 
$$ 
where $\alpha^{[\frac{m+1}2:2]\pm}_{\frac{m+1}2}$, 
which was used to subdivide $M'$, 
has been moved and ``doubled'' to $\alpha^{0+}_0 \cup \alpha_{m-1}^{0+} \subset \partial M'''$. 

We have then that a first eigenfunction in $W_-$ (corresponding to $\lambda_1(W_-)$) 
corresponds to an eigenfunction 
in $C_{pw}^\infty [M'''; \alpha^{0+}_0 \cup \alpha_{m-1}^{0+} \cup (\gamma_{4\pm}\cap M''') , (\gamma_{4\mp}\cap M''') \, ]$  
which moreover is even under reflection with respect to  
$\refl_{ \pqq_{m/2} , C^\perp \, }  = \refl_{\Sigma_{\frac\pi2 - \frac\pi{2m} } }$, 
and where the $\pm$ and $\mp$ signs are opposite and depend on whether $m\in 4\Z+1$ or $m\in 4\Z+3$.  
Either way by \ref{P:V+--} and by domain monotonicity 
(since $M'''\subsetneq M^{++}_{+*}$)  
we conclude 
that $\lambda_1(W_-) >0$.  

We have also $W_+\sim C_{pw}^\infty [ M''; (\gamma_{4\pm}\cap M'') , \alpha^{0+}_0 \cup \alpha_{\frac{m-1}2}^{[\frac{m-1}2:2]\pm}  \cup (\gamma_{4\mp}\cap M'') \, ]$.  
Suppose $\varphi$ is an eigenfunction corresponding to 
$\lambda_2(W_+)$.   
By Courant's nodal theorem \ref{Ctheorem} and arguing as in the proof of \ref{V+++} using 
\cite{Cheng}*{Theorem 2.5}, we conclude that 
there is a 
separating nodal curve $\gamma$ which has to avoid at least one of
$\gamma_4\cap M''$, $\alpha^{0+}_0 $, or $ \alpha_{\frac{m-1}2}^{[\frac{m-1}2:2]\pm} $.  
In the first case by domain monotonicity we conclude that 
$$ 
\lambda_2(W_+) > 
\lambda_1 [ M''; (\gamma_{4}\cap M'') , \alpha^{0+}_0 \cup \alpha_{\frac{m-1}2}^{[\frac{m-1}2:2]\pm}  \, ] =0, 
$$ 
where the last equality follows from \ref{P:V--}. 
In the second case we again use domain monotonicity, 
but the comparison is with $\lambda_1(W_-) $, 
which we proved positive above. 
In the third case we reposition $M''$ and we argue as for the second case. 
This completes the proof that 
$\lambda_2(W_+) > 0$ and hence of our claim 
\ref{claim}. 

Clearly by \ref{P:V--} we have 
\begin{equation}
\label{E:D1}
\lambda_1 [ D_{0+}^{0+} ; \beta^0_{0+} , \overline{ \pqq_1\pqq^1 } \cup \alpha^{0+}_0 ] = 0. 
\end{equation} 
We consider now an eigenfunction corresponding to 
$\lambda_2 [ D_{1-}^{1-} ; \emptyset,  \beta^1_{1-} \cup \overline{ \pqq_1\pqq^1 } \cup \alpha^{1-}_1] $. 
By Courant's nodal theorem \ref{Ctheorem} and arguing as in the proof of \ref{V+++} using 
\cite{Cheng}*{Theorem 2.5} again, we conclude that 
there is a separating nodal curve which avoids at least one of 
$\beta^1_{1-}$, $\overline{ \pqq_1\pqq^1 }$,  or $\alpha^{1-}_1$. 
We can use domain monotonicity then to assert that 
$\lambda_2 [ D_{1-}^{1-} ; \emptyset,  \beta^1_{1-} \cup \overline{ \pqq_1\pqq^1 } \cup \alpha^{1-}_1] $ 
is $>$ one of 
$\lambda_1 [ D_{1-}^{1-} ; \beta^1_{1-} , \overline{ \pqq_1\pqq^1 } \cup \alpha^{1-}_1] $, 
$\lambda_1 [ D_{1-}^{1-} ; \overline{ \pqq_1\pqq^1 } , \beta^1_{1-} \cup \alpha^{1-}_1] $, 
or 
$\lambda_1 [ D_{1-}^{1-} ; \alpha^{1-}_1 , \beta^1_{1-} \cup \overline{ \pqq_1\pqq^1 } ] $.  
By appealing to \ref{P:V--}, \ref{V+++}, or \ref{V++-} correspondingly we conclude that 
\begin{equation}
\label{E:D2}
\lambda_2 [ D_{1-}^{1-} ; \emptyset,  \beta^1_{1-} \cup \overline{ \pqq_1\pqq^1 } \cup \alpha^{1-}_1] >0.  
\end{equation} 

We can apply \ref{P:eigen} now to the decomposition 
$M^{++}_{+*} = D_{0+}^{0+} \cup D_{1-}^{1-} \cup M'$ to conclude 
by referring to \ref{claim}, \ref{E:D1}, and \ref{E:D2}, that 
$\#_{\le0}(V_{+}^{-+}) = \#_{\le0}[M^{++}_{+*}; \gamma_{4-} , \gamma_{4+} \cup \gamma_5 ] \le2$.  
Since 
$J_{C_{\pi/2}^{0}}\in V_{+}^{-+}$ changes sign on $M^{++}_{+*}$ by \ref{L:grad}, 
it cannot be a first eigenfunction. 
Since it has eigenvalue $0$, we conclude by the last inequality that
it is a second eigenfunction, which completes the proof of (iii) for $V_+^{-+}$. 
The proof for $V_{+}^{+-}$ is similar. 
\end{proof} 

The main theorem follows. 
Recall that $\xi_{g,1}$ in the notation of \cite{Lawson} denotes the genus-$g$ Lawson surface 
which can be viewed as a desingularization of two orthogonal great two-spheres in the round three-sphere $\Sph^3$.   

\begin{theorem} 
\label{Mtheorem} 
If $g\in\N$ and $g\ge2$, then the index of $\xi_{g,1}$ is $2g+3$.  
Moreover $\xi_{g,1}$ has nullity $6$ and no exceptional Jacobi fields. 
\end{theorem}

\begin{proof} 
Recall that $m=g+1$.  
Combining then \ref{V++}, \ref{V+-+o}, 
Propositions \ref{P:V--}, \ref{P:V++}, \ref{P:V+--} and \ref{P:V+-+}, 
we conclude the proof. 
\end{proof} 

\begin{remark}[Alternative proof for high genus]  
\label{R2} 
The Lawson surfaces of high genus can be constructed by gluing and then one obtains a detailed knowledge of their geometry.  
The gluing construction is a straightforward desingularization construction for $\Sigma^{\pi/4}\cup \Sigma^{-\pi/4}= \cup_{j=1}^4 (\pqq^j\cone C)$, 
that is two orthogonal great two-spheres, 
in the fashion of those constructions in \cite{kapouleas:wiygul} which are for two orthogonally intersecting Clifford tori. 
The surfaces constructed are modeled in the vicinity of $C$ after the classical Scherk surface \cite{Scherk} desingularizing two orthogonal planes in $\R^3$  
and 
given in appropriate Cartesian coordinates 
by the equation $\sinh x^1 \sinh x^2 = \sin x^3$. 
For each large $m$ we can impose on the construction all the symmetries of $M=M[C,m]$.  
By the uniqueness then in \ref{T:lawson} we can infer that the surface constructed is actually $M=M[C,m]=\xi_{m-1,1}$.  
By the control the construction provides we can conclude then that for large $m$ (equivalently large genus) 
the region of the Lawson surface $M=\xi_{m-1,1}$ in the vicinity of $C$ can be approximated by an appropriately scaled singly periodic Scherk surface,   
which has been transplanted to $\Sph^3$ so that its axis covers $C$. 
The rest of the Lawson surface approximates $\cup_{j=1}^4 (\pqq^j\cone C)$ (that is the two great two-spheres being desingularized) 
with a small neighborhood of $C$ removed. 
This information can be used to simplify the proofs of many intermediate results on which the proof of the main theorem is based, 
thus avoiding the need for many of the arguments we have used in this article. 
\qed 
\end{remark} 

Note now that there is a smooth family of singly periodic Scherk surfaces which can be parametrized by the angle $\theta\in(0,\pi)$ between 
two adjacent asymptotic half planes. 
The Scherk surface in \ref{R2} corresponds then to $\theta=\pi/2$. 
Since we can then prescribe $\theta$, we say that the Scherk surfaces can \emph{``flap their wings''}. 
In \cite{alm20}*{Section 4.2} a heuristic argument was provided indicating that this is not the case for the Lawson surfaces 
and further questions motivated by this were asked.  
The non-existence of exceptional Jacobi fields as in \ref{Mtheorem} provides partial answers to some of those questions. 
In particular it implies that each $\xi_{g,1}$ is isolated as in the following corollary. 
Isolatedness can be proved by adapting the proof of \cite{KMP}*{Proposition 3.1} or by a more direct argument suggested to us by R. Schoen as follows.

\begin{cor}[No flapping and isolatedness] 
\label{C:flapping} 
$\xi_{g,1}$ as in \ref{Mtheorem} 
cannot ``flap its wings'' at the linearized level and moreover it is isolated in the sense that there is an $\epsilon>0$ such that 
any minimal surface 
within a $C^1$ $\epsilon$-neighborhood of $\xi_{g,1}$  
is congruent to $\xi_{g,1}$.  
\end{cor} 

\begin{proof} 
By ``no flapping at the linearized level'' we mean that there are no Jacobi fields 
which are infinitesimal deformations consistent with varying the angle of intersection of the 
two spheres of which the surfaces can be viewed as desingularizations. 
Since there are no exceptional Jacobi fields by Theorem \ref{Mtheorem}, 
the result follows. 

To prove now isolatedness 
suppose $M=\xi_{g,1}$ is not isolated modulo congruence.
Then there exists a sequence $\{M_n\}$
of embedded 
minimal surfaces none of which is congruent to $M$
but which $C^1$-converge to $M$.
Each $M_n$ is then the graph (via the exponential map in the normal direction)
of some function $u_n$ on $M$, so that $\{u_n\}$ converges to $0$ in $C^1(M)$.
By appropriately rotating each $M_n$ we may assume that eventually $u_n$
is $L^2(M)$-orthogonal to the space of nonexceptional Jacobi fields,
at least to first order in $\norm{u_n}_{C^0}$.
Since each $M_n$ is minimal, by elliptic regularity
the sequence $\left\{u_n/\norm{u_n}_{C^1}\right\}$  
is bounded in $C^{2,\alpha}(M)$, so has a subsequence converging in $C^2(M)$,
thereby producing a nontrivial exceptional Jacobi field, a contradiction.
\end{proof} 

\appendix

\section{Eigenvalues and subdivisions} 
\label{A:eigen}
\nopagebreak
  
In this appendix following \cite{Ros} we state two bounds on the number of eigenvalues on a domain 
in terms of the number of eigenvalues on appropriate subdivisions of the domain. 
More precisely suppose that we are given $\Lcalu$, $\Uu$, $\gu$, and $\partial \Uu= \partial_D\Uu \cup \partial_N\Uu$ as in \ref{D:mixed}. 
We assume further that by removing a finite union of smooth embedded one-dimensional submanifolds, 
$\gammau\subset \Uu$, 
we subdivide $\Uu$ into $n\in\N$ connected components 
whose (compact) closures we denote by $\Uu_i$ for $i=1,\dots,n$. 
We define $\partial_D \Uu_i:= \partial \Uu_i \cap \partial_D\Uu$,  
$\partial_N \Uu_i:= \partial \Uu_i \cap \partial_N\Uu$,  
and 
$\gammau_i:= \partial \Uu_i \cap \gammau$.  
Clearly then we have the decomposition 
$$ 
\partial \Uu_i = \gammau_i \cup \partial_D \Uu_i \cup \partial_N \Uu_i.  
$$
We have then the following. 

\begin{prop}[Montiel-Ros \cite{Ros}]  
\label{P:eigen} 
Assuming the above and in the notation of \ref{D:mixed} we have the following $\forall\lambda\in\R$. 
\\
(i)  
$\#_{<\lambda}[\Lcalu,\Uu;\partial_D\Uu,\partial_N\Uu]
\ge 
\#_{<\lambda}[\Lcalu,\Uu_1;\gammau_1\cup \partial_D\Uu_1,\partial_N\Uu_1] 
+
\sum_{i=2}^n
\#_{\le\lambda}[\Lcalu,\Uu_i;\gammau_i\cup \partial_D\Uu_i,\partial_N\Uu_i]$. 
\\
(ii)  
$\#_{\le\lambda}[\Lcalu,\Uu;\partial_D\Uu,\partial_N\Uu]
\le 
\#_{\le\lambda}[\Lcalu,\Uu_1;\partial_D\Uu_1,\gammau_1\cup \partial_N\Uu_1] 
+
\sum_{i=2}^n
\#_{<\lambda}[\Lcalu,\Uu_i;\partial_D\Uu_i,\gammau_i\cup \partial_N\Uu_i]$. 
\end{prop} 

\begin{proof} 
Parts (i) and (ii) generalize Lemma 12 and Lemma 13 respectively of \cite{Ros},
whose proofs carry over here with only minor modification.
Nevertheless we sketch a proof for ease of reference.
First we introduce some general notation:
for  $\Lcalu$, $\Uu$, $g$, and $\partial \Uu= \partial_D\Uu \cup \partial_N\Uu$ as in \ref{D:mixed}
and $\lambda \in \R$
we will write $E_\lambda[\underline{U};\partial_D\underline{U},\partial_N\underline{U}]$
for the $\lambda$-eigenspace of $\Lcalu$ on $\underline{U}$
with Dirichlet condition on $\partial_D \underline{U}$ and Neumann condition on $\partial_N\underline{U}$.
We will understand $E_\lambda[\Uu;\partial_D\Uu,\partial_N\Uu]=\{0\}$
when $\lambda$ is not an eigenvalue.
Now we make the same assumptions on $\Uu$ and its boundary as above
and fix $\lambda \in \R$.

To prove (i) we define the $n$ spaces of test functions
  \begin{equation*}
  \begin{aligned}
    V_1
    &:=
    \{
      u \in C^\infty_{pw}[\Uu] \, : \,
      u|_{\Uu_1} \in \bigoplus_{\lambda'<\lambda} E_{\lambda'}[\Uu_1;\gammau_1 \cup \partial_D\Uu_1,\partial_N\Uu_1]
      \mbox{ and } u|_{\Uu_j}=0 \mbox{ if } j \neq 1
    \} \mbox{ and} \\
    V_i
    &:=
    \{
      u \in C^\infty_{pw}[\Uu] \, : \,
      u|_{\Uu_i} \in \bigoplus_{\lambda' \leq \lambda} E_{\lambda'}[\Uu_i;\gammau_i \cup \partial_D\Uu_i,\partial_N\Uu_i]
      \mbox{ and } u|_{\Uu_j}=0 \mbox{ if } j \neq i
    \} \mbox{ for } 2 \leq i \leq n.
  \end{aligned}
  \end{equation*}
Clearly for any $u_i \in V_i$ and $u_j \in V_j$ with $i \neq j$
we have $\nabla u_i \perp_{L^2(\Uu)} \nabla u_j$
and $u_i \perp_{L^2(\Uu)} hu_j$ for any $h \in C^\infty_{pw}(\underline{U})$.
Define also 
  \begin{equation*}
    V_{<\lambda}
    :=
    \bigoplus_{\lambda'<\lambda} E_{\lambda'}[\Uu;\partial_D\Uu,\partial_N\Uu].
  \end{equation*}
We claim that the $L^2(\Uu)$-orthogonal projection
$\bigoplus_{i=1}^n V_i \to V_{<\lambda}$ is injective,
which will establish (i).
To check the injectivity suppose $u \in \bigoplus_{i=1}^n V_i$, so that
$u=\sum_{i=1}^n u_i$
for some $u_1 \in V_1$, $u_2 \in V_2$, \ldots, $u_n \in V_n$.
Then
  \begin{equation*}
      \langle \nabla u, \nabla u \rangle_{L^2} - \langle u,fu\rangle
      =
      \sum_{i=1}^n \left(\langle \nabla u_i, \nabla u_i \rangle_{L^2} - \langle u_i, fu_i \rangle_{L^2}\right)
      \leq
      \lambda \langle u, u \rangle_{L^2},
  \end{equation*}
but the additional assumption $u \perp_{L^2} V_{<\lambda}$
forces the equality case,
which then implies
$u|_{\underline{U}_1}=0$ and $\Lcalu u=-\lambda u$ everywhere.
We conclude by the unique-continuation principle \cite{Aronszajn} that in fact $u=0$.

To prove 
now (ii) we define the $2+n$ vector spaces
  \begin{equation*}
  \begin{aligned}
    &W:=\prod_{i=1}^n C^\infty_{pw}[\Uu_i], \quad
    V_{\leq \lambda}
  	:=
  	\bigoplus_{\lambda' \leq \lambda} E_{\lambda'}[\Uu,\partial_D\Uu,\partial_N\Uu],
  	\quad
  	W_1
  	:=
  	\bigoplus_{\lambda' \leq \lambda} E_{\lambda'}[\Uu_1,\partial_D\Uu_1,\gammau_1 \cup \partial_N\Uu_1], \\
    &\mbox{and for } 2 \leq i \leq n \qquad
    W_i
    :=
    \bigoplus_{\lambda'<\lambda} E_{\lambda'}[\Uu_i,\partial_D\Uu_i,\gammau_i \cup \partial_N\Uu_i].
  \end{aligned}
  \end{equation*}
Clearly $\prod_{i=1}^n W_i$ is a subspace of $W$
and the map
$\iota: u \in V_{\leq \lambda} \mapsto \left(u|_{\Uu_1},\ldots,u|_{\Uu_n}\right) \in W$
is injective.
Endowing $W$ with the obvious $L^2$ inner product and writing
$\pi: W \to \prod_{i=1}^n W_i$
for the corresponding projection onto $\prod_{i=1}^n W_i$,
we claim that $\pi \circ \iota$ is injective,
implying (ii).
To check the injectivity suppose $u \in V_{\leq \lambda}$.
Then
 \begin{equation*}
   \sum_{i=1}^n \langle \nabla u|_{\underline{U}_i}, \nabla u|_{\underline{U}_i} \rangle_{L^2}
      -\sum_{i=1}^n \langle fu|_{\underline{U}_i},u|_{\underline{U}_i} \rangle_{L^2}
    =
    \langle \nabla u, \nabla u \rangle_{L^2} - \langle u, fu \rangle_{L^2}
    \leq \lambda \langle u, u \rangle_{L^2}.
  \end{equation*}
On the other hand, if $\iota(u) \perp \prod_{i=1}^n W_i$,
then for each $i$
  \begin{equation*}
    \langle \nabla u|_{\underline{U}_i}, \nabla u|_{\underline{U}_i} \rangle_{L^2}
      -\langle fu|_{\underline{U}_i},u|_{\underline{U}_i} \rangle_{L^2}
    \geq
    \lambda \langle u|_{\underline{U}_i}, u|_{\underline{U}_i} \rangle_{L^2},
  \end{equation*}
with strict inequality when $n=1$, unless $u|_{\Uu_1}=0$.
Since $\sum_{i=1}^n \langle u|_{\underline{U}_i}, u|_{\underline{U}_i} \rangle_{L^2}=\langle u,u\rangle_{L^2}$,
we conclude that if $u \in \ker \pi \circ \iota$, then $u$ is a solution to $\underline{\mathcal{L}}u=-\lambda u$ 
and vanishes identically on $\underline{U}_1$,
forcing $u=0$ by unique continuation.
\end{proof}

\section{The Courant nodal theorem} 
\label{B:courant}
\nopagebreak
  
In this appendix we recall Courant's nodal theorem in the form we use it. 
Suppose that we are given $\Lcalu$, $\Uu$, $\gu$, and $\partial \Uu= \partial_D\Uu \cup \partial_N\Uu$ as in \ref{D:mixed}. 
Suppose moreover $\Uu$ is connected. 
We define the number of nodal domains of an eigenfunction $u$ of $\Lcalu$ to be the number of connected components of 
$\Uu\setminus u^{-1}(0)$. 
We have then the following, 
where for ease of reference we include in the theorem its corollary on the simplicity of the first eigenvalue. 

\begin{theorem}[Courant's nodal theorem \cite{courant}]  
\label{Ctheorem} 
Given $\Lcalu$, $\Uu$, $g$, and $\partial \Uu= \partial_D\Uu \cup \partial_N\Uu$ as above, 
let $N_n$ for each $n\in\N$ 
be the number of nodal domains of an eigenfunction corresponding to the $n^{th}$ eigenvalue 
$\lambda_n[\Lcalu,\Uu;\partial_D\Uu,\partial_N\Uu]$   
in the notation of \ref{D:mixed}.  
We have then for $n=1$: $N_1=1$ and 
$\lambda_1[\Lcalu,\Uu;\partial_D\Uu,\partial_N\Uu]  <  
\lambda_2[\Lcalu,\Uu;\partial_D\Uu,\partial_N\Uu] $;   
and for $n>1$: $2\le N_n \le n$. 
\end{theorem} 

\def\baselinestretch{.9}

\bibliographystyle{amsplain}
\bibliography{paper}
\end{document}